\documentclass[11pt,a4paper]{article}
\usepackage[T1]{fontenc}
\usepackage[utf8]{inputenc}
\usepackage{lmodern}
\usepackage[margin=27mm]{geometry}
\usepackage{amsmath,amssymb,amsthm,mathtools,mathrsfs}
\usepackage{microtype,enumitem,booktabs,tabularx,needspace}
\usepackage{xcolor}
\usepackage{caption}
\definecolor{ReferenceBlue}{RGB}{0,70,170}
\usepackage{tikz}
\usepackage{tikz-cd}
\usetikzlibrary{patterns}
\usetikzlibrary{positioning}
\usepackage{hyperref}
\hypersetup{
  colorlinks=true,
  linkcolor=ReferenceBlue,
  citecolor=ReferenceBlue,
  urlcolor=ReferenceBlue,
  linktoc=all
}
\usepackage{comment}
\usepackage{bookmark}
\usepackage[nameinlink,noabbrev,capitalise]{cleveref}
\newcommand{\namedref}[2]{\hyperref[#2]{#1~\ref*{#2}}}
\renewcommand*{\eqref}[1]{\hyperref[#1]{\textup{\normalfont(\ref*{#1})}}}
\newtheorem*{mainthm}{Theorem} 
\newtheorem{theorem}{Theorem}[section]
\newtheorem{proposition}[theorem]{Proposition}
\newtheorem{lemma}[theorem]{Lemma}
\newtheorem{corollary}[theorem]{Corollary}
\theoremstyle{definition}
\newtheorem{definition}[theorem]{Definition}

\theoremstyle{remark}
\newtheorem{remark}[theorem]{Remark}
\numberwithin{equation}{section}
\setlist[enumerate]{label=\textup{(\roman*)},leftmargin=2.1em,itemsep=2pt,topsep=4pt}
\setlist[itemize]{leftmargin=1.8em,itemsep=2pt,topsep=4pt}
\newcommand{\cF}{\mathcal F}
\newcommand{\cT}{\mathcal T}

\newcommand{\N}{\mathbb N}
\newcommand{\Nzero}{\mathbb N_0}
\newcommand{\OO}{\mathcal O}
\newcommand{\PP}{\mathcal P}
\newcommand{\Rbar}{\overline{\mathbb R}_{+}}
\newcommand{\id}{\operatorname{id}}
\newcommand{\Sc}{\Sigma}
\newcommand{\Law}{\lambda}
\newcommand{\up}{\mathord\uparrow}
\newcommand{\dn}{\mathord\downarrow}
\newcommand{\Up}{\mathord\Uparrow}
\newcommand{\Down}{\mathord\Downarrow}
\newcommand{\DD}{\mathord\Downarrow}

\newcommand{\FinSet}{\mathbf{FinSet}}

\newcommand{\intr}{\operatorname{int}}
\newcommand{\Cl}{\operatorname{cl}}
\newcommand{\Idl}{\operatorname{Idl}}

\newcommand{\Min}{\operatorname{Min}}
\newcommand{\mub}{\operatorname{mub}}
\newcommand{\Komp}{\operatorname{K}}
\newcommand{\Cov}{\operatorname{Cov}}
\newcommand{\Lab}{\operatorname{lab}}
\newcommand{\supp}{\operatorname{supp}}
\newcommand{\guard}{\operatorname{guard}}

\newcommand{\FBR}{\mathrm{FBR}}
\newcommand{\LUC}{\mathrm{LUC}}
\newcommand{\Test}{\mathbb X^{\top}_{\bot}}
\newcommand{\tops}{\mathbf 1_Z}
\newcommand{\bots}{\mathbf 0_Z}
\newcommand{\doi}[1]{\href{https://doi.org/#1}{doi: #1}}
\newcommand{\intrs}{\operatorname{int}_{\sigma}}

\hypersetup{pdftitle={Function spaces, web spaces, and Lawson-compact algebraic retracts},pdfauthor={Wei Luan}}
\title{Function spaces, web spaces, and Lawson-compactness in Domain Theory}

\author{Wei Luan$^{a}$, Qingguo Li$^{b}$\thanks{Corresponding author.}\\[3pt]
	\parbox{0.95\textwidth}{\centering
		\small $^{a}$ School of Mathematics and Statistics, Hunan Normal University, Changsha, Hunan 410081, P.~R.~China\\
		\small $^{b}$ School of Mathematics, Hunan University, Changsha, Hunan 410082, P.~R.~China\\[2pt]
		\small E-mail addresses: \href{mailto:luanwei@hunnu.edu.cn}{luanwei@hunnu.edu.cn} (W. Luan),\\
		\small \href{mailto:liqingguoli@aliyun.com}{liqingguoli@aliyun.com} (Q. Li)}}
	
\date{}
\begin{document}
\maketitle
\begin{abstract}
This paper investigates the structural properties of function spaces and their connections with Lawson-compactness and web spaces in domain theory. First, using finite-refinement and palette methods, we characterize the topological spaces \(X\) for which the function space \([X\to L]\) is an RB-domain for every RB-domain \(L\). In particular, we establish the equivalence between this preservation property and the condition that \(X\) is compact, core compact, core coherent, and a web space. We further investigate its relationship with RW$^*$-spaces and obtain partial results concerning the Lawson–Mislove's problem. Second, we construct a countably based topological space \(Z\) that is compact, core compact, core coherent, and a web space, but is neither a wide web space nor a C-space. More precisely, its lattice of open sets is a continuous frame satisfying \(\kappa\)-distributivity for every finite cardinal \(\kappa\geq 2\), while failing to be \(\omega\)-distributive. This provides a negative answer to a question raised by Ern\'e and demonstrates that finite distributivity, even at every finite level, does not imply countable distributivity. Finally, by introducing labelled posets and employing constructions based on presheaves over finite sets and their element posets, we prove that every Lawson-compact domain is a Scott-continuous retract of a Lawson-compact algebraic domain, thereby answering a problem proposed by Kou. As applications, we show that the function space from an RB-domain to a Lawson-compact domain is Lawson-compact. Moreover, for an FS-domain \(X\), we establish that \(X\) is an RB-domain if and only if \([X\to L]\) is Lawson-compact for every Lawson-compact domain \(L\). This characterization also yields the existence of an FS-domain \(X\) and a Lawson-compact algebraic domain \(L\) for which \([X\to L]\) is not continuous. This gives a negative answer to the question raised by Shen et al. 
\end{abstract}
\clearpage
{\small
\renewcommand{\contentsname}{\textcolor{ReferenceBlue}{Contents}}
\tableofcontents}

\section{Introduction}

Domain theory originated in Dana Scott's order-theoretic approach to the denotational semantics of programming languages, where partially ordered sets are used to describe information and its successive approximations. The fundamental idea is that elements of a domain may be approximated by simpler or more finite pieces of information, and that the order structure records the increase of information during computation. This viewpoint has led to a rich mathematical theory connecting order theory, topology, lattice theory, and theoretical computer science. Continuous domains, in particular, provide a natural framework in which approximation can be expressed through the way-below relation, while the Scott-topology translates the underlying order-theoretic structure into topological terms. The interaction between order and topology has consequently become one of the central themes of domain theory \cite{AbramskyJung1994,CLD}.

The Lawson-topology, arising from Lawson's work on ordered topological structures and topological semilattices \cite{Law73}, is the join of the Scott-topology and the lower topology. On a continuous domain, it is Hausdorff and provides a natural bridge between order-theoretic approximation and classical topological methods. Continuous lattices are classical examples of Lawson-compact domains, while FS-domains and RB-domains constitute important classes possessing Lawson-compactness through their finite approximation structures. The class of Lawson-compact domains is therefore broad enough to include many domains arising naturally in semantics, yet sufficiently structured to support representation theorems and function-space constructions. This paper mainly investigates properties of function spaces of Lawson-compact domains and explores their connections with web spaces.

\quad

\noindent\textbf{1.1 Function-space calculus and Lawson-Mislove's problem.} An important direction in the development of domain theory concerns function spaces. It is closely connected with Cartesian closedness: a Cartesian closed category of domains is closed under products and exponentials, and therefore provides a natural setting in which higher-order function types and recursive domain equations can be interpreted. Liu and Liang prove that if $\mathcal{A}$ is the class of all core compact spaces, then the largest class $\mathcal{B}$ of dcpos such that $[A\rightarrow B]$ is a continuous domain for any $A\in \mathcal{A}$ and $B\in \mathcal{B}$ is the class of continuous L-domains \cite{LiuLiang1996}. Correspondingly, if $\mathcal{A}$ is the class of all compact, core compact spaces, then the corresponding $\mathcal{B}$ is the class of all continuous $sL$-domains \cite{LawsonXu2003'}. Such problems are explicitly proposed in \cite{LawsonXu2003'}:

\begin{quote}
		\itshape
	Given a class or collection of domains (or topological spaces) $\mathcal{A}$, find the largest class of domains $\mathcal{B}$ such that the space of continuous functions $[A\rightarrow B]$ is again a continuous domain (with respect to the pointwise order) for each $A\in \mathcal{A}$ and each $B\in \mathcal{B}$. If this problem is solved, then one can consider the problem of minimizing the class $\mathcal{A}$ that gives the same solution $\mathcal{B}$. One can also study the analogous problem with the roles of $\mathcal{A}$ and $\mathcal{B}$ reversed.
\end{quote}
In the study of the above problem, one may substitute other properties for the continuity of function spaces, for instance, Lawson-compactness, as proposed by Lawson and Xu \cite{LawsonXu2003}. We refer to this class of problems as ``function-space calculus''. The research related to Lawson-compactness originated from the work of Kou and Luo on RW-spaces \cite{KouLuo2003}. It is proved that a core compact space $X$ is a RW-space if and only if $[X\to L]$ is Lawson-compact for every pointed compact continuous L-domain $L$. According to results of Liang and Keimel \cite{LiangKeimel1999} (see also \cite{ZouLiGuo2018} by Zou, Li and Guo), every Lawson-compact L-domain is an RB-domain. A natural question then arises: if one replaces ``Lawson-compact L-domain'' with ``RB-domain'', what is the resulting class of topological spaces? Following this line of investigation, Luan and Li \cite{LuanLiSG} studied the above problem in the setting of algebraic domains. It is proved that a space $X$ satisfies that $[X\to L]$ is a Lawson-compact algebraic domain for all bifinite domain $L$ if and only if $X$ is compact, core compact and core coherent spaces whose open set lattice $\OO(X)$ is supercompactly generated. In the continuous domain case, a likely candidate seems to consist of ``compact, core compact, core coherent spaces whose open set lattices are completely distributive''. Nevertheless, this problem has hitherto been unresolved.

Somewhat surprisingly, the answer to this problem is closely connected with the web spaces introduced by Ern\'e \cite{Erne2009}. These spaces was introduced to study infinitary distributive laws and their relationships with local connectedness and compactness properties of topological spaces. The main result of this part is:

\begin{mainthm}
	For a topological space $X$, the following are equivalent:
	\begin{enumerate}
		\item $X$ is compact, core compact, core coherent, and a web space;
		\item The sobrification of the $T_0$-reflection of $X$ is a stably compact web space;
		\item $[X\to L]$ is an RB-domain for every RB-domain $L$;
		\item $[X\to L]$ is a Lawson-compact domain for every RB-domain $L$;
		\item $[X\to\mathbf2]$ and $[X\to\Lambda]$ are continuous domains,
		and $[X\to\Test]$ is a Lawson-compact domain.
	\end{enumerate}
\end{mainthm}
One major motivation for carrying out function-space calculus using the Lawson-compactness of function spaces is to tackle the Lawson–Mislove's problem. This refers to Problem 544 in the book \emph{Open Problems in Topology} \cite{OpenProbTop90}:
\begin{quote}
		\itshape
	Characterize those pairs $(X, P)$ such that $X$ is a core compact space, $P$ is a continuous DCPO, and $[X\to P]$ is a continuous DCPO for
	which the $\lambda$-topology is compact.
\end{quote}
With a view to using these results to resolve the Lawson–Mislove's problem, we then introduce the notion of an RW$^*$-space as a subclass of RW-spaces. A space $X$ is an RW$^*$-space if and only if, for every Lawson-compact $sL$-domain $L$, the function space $[X\to L]$ is a Lawson-compact domain. Using this result, together with the forbidden-substructure argument for pointed RB-domains $L$ that fail to be $sL$-domains, we resolve the Lawson–Mislove's problem within the scope of $\mathcal{RW}^*\times \mathcal{RB}_\bot$ (see \namedref{Table}{tab:mytable}). Here $\mathcal{RW}^*$ is the family of all RW$^*$-spaces and $\mathcal{RB}_\bot$ is the class of all pointed RB-domains.

\quad

\noindent\textbf{1.2 Connection with Ern\'e’s problem and a solution.} As mentioned above, a likely candidate of a space $X$ for which $[X\to L]$ is a Lawson-compact domain for all RB-domain $L$ seems to consist of ``compact, core compact, core coherent spaces whose open set lattices are completely distributive''. Yet, as it happens, this class of spaces contains ``compact, core compact, core coherent web spaces''. Accordingly, the key question now is whether the open set lattice of a web space is necessarily completely distributive (or equivalently, the space is a $C$-space), provided that the space is compact, core compact, and core coherent. 

Since, under the assumption of core compactness, a space is a web space if and only if its lattice of open sets is $3$-distributive, whereas complete distributivity is equivalent to continuity together with $\omega$-distributivity, this is precisely the open problem raised by Ern\'e in \cite{Erne2009}:
\begin{quote}
		\itshape
	But it remains open whether all $4$-web spaces are already
	$\omega$-web spaces, and whether all $3$-distributive lattices are
	$\omega$-distributive ($=\mathcal F$-distributive).
\end{quote} 
We answer Ern\'e's question in the negative. More precisely, we construct a countably based space $Z$ that is compact, core compact, core coherent, and a web space, but is neither a wide web space nor a $C$-space (\namedref{Theorem}{thm:counterexample}). Its open set lattice is a continuous frame that is $\kappa$-distributive for every finite cardinal $\kappa\geq 2$, but fails to be $\omega$-distributive. Thus finite distributivity, even at every finite level, does not imply the corresponding infinitary distributivity. In particular, the conjectured $C$-space characterization of preservation of RB-domains fails. At the same time, the example still has the preservation property established in the first part: $[Z\to L]$ is an RB-domain for every RB-domain $L$. This separates the finite web and coframe conditions from the stronger complete-distributivity conditions underlying $C$-spaces.

\quad

\noindent\textbf{1.3 Retracts of Lawson-compact algebraic domains.} The classification of Cartesian closed categories of domains has long been a fundamental problem in domain theory, particularly in connection with the denotational semantics of higher-order programming languages. Building on earlier work of Smyth \cite{Smyth1983}, Jung undertook a systematic investigation of this problem and obtained a complete classification of the maximal Cartesian closed full subcategories of algebraic domains \cite{Jung1988}. In particular, he showed that, in the pointed setting, there are exactly two such maximal categories, namely the category of bifinite domains and that of algebraic L-domains. Subsequently, Jung extended his investigation to the broader setting of continuous domains and introduced the class of FS-domains. His classification theorem established that the categories of FS-domains and continuous L-domains are precisely the two maximal Cartesian closed full subcategories of pointed continuous domains. An important class closely related to FS-domains is that of RB-domains, which are defined as Scott-continuous retracts of bifinite domains. It is known that every RB-domain is an FS-domain, while continuous L-domains are precisely the Scott-continuous retracts of algebraic L-domains. This naturally led to the longstanding question of whether an analogous characterization holds for FS-domains: ``Is every FS-domain a Scott-continuous retract of a bifinite domain?'' Equivalently, do the classes of FS-domains and RB-domains coincide? This question remained open for more than three decades. Recently, in 2026, Chen, Kou, and Lyu \cite{ChenKouLyu2026} resolved it in the negative by proving that Lawson's planar closed-disk domain is an FS-domain but not an RB-domain. For related work, see also \cite{Abb26,Che26}. Their result establishes that the class of RB-domains is a proper subclass of the class of FS-domains, thereby settling a longstanding problem arising from Jung's classification theory.

An FS-domain is a Lawson-compact domain, while a bifinite domain is a Lawson-compact algebraic domain. During \emph{The Fifth National Young Scholars Forum on Domain Theory and Topology}, Hui Kou put forward the following open problem: 
\begin{quote}
	\itshape
	Is every Lawson-compact domain a Scott-continuous retract of some Lawson-compact algebraic domain?
\end{quote}
This problem is independent of the above question whether every \(\mathrm{FS}\)-domain is a retract of a bifinite domain; the answer to neither implies the other directly. The significance of this problem lies in the following: if the conjecture holds, then results concerning function spaces of Lawson-compact algebraic domains can be transferred directly to Lawson-compact domains. For instance, Lawson and Xi showed in \cite{LawsonXi2015} that the function space of maps from a coherent algebraic domain to a bifinite domain with a least element is bifinite. Nevertheless, it is not easy to prove that the function space from a coherent domain to a pointed RB-domain is an RB-domain \cite{XiXuLawson2016}. Furthermore, Luan and Li \cite{LuanLi2025} proved that the function space from a bifinite domain to a Lawson-compact algebraic domain is Lawson-compact algebraic, while it remains unknown whether the function space from an RB-domain to a Lawson-compact domain is Lawson-compact.

In this paper, we introduce labelled posets. Given any Lawson-compact domain $D$, we construct a labelled poset $P$ with property M from a presheaf over finite sets and its element poset so that $D$ is a Scott-continuous retract of the ideal completion of $P$. This yields the result that every Lawson-compact domain is a Scott-continuous retract of a Lawson-compact algebraic domain.

As for applications of this result, it first implies that the function space from an RB-domain to a Lawson-compact domain is Lawson-compact. In addition, Chong Shen et al. raised the following question in \cite[Question 8.1]{ShenHoXiZhao2026}:
\begin{quote}
	\itshape
	Let $X$ be an FS-domain and let $L$ be a continuous domain whose Lawson topology is compact. Must both Scott function spaces $[X\to L]$ and $[L\to X]$ be continuous?
\end{quote}
It is shown that for an FS-domain $X$, $X$ is an RB-domain if and only if $[X\to L]$ is a Lawson-compact domain for all Lawson-compact domain $L$. An immediate corollary of this result is that one can construct an FS-domain $X$ and a Lawson-compact algebraic domain $L$ such that the function space $[X\to L]$ fails to be a continuous domain, which answers the question raised by Chong Shen et al. This also re-solves the problem discussed by Abramsky and Jung \cite[Proposition 4.2.10]{AbramskyJung1994}: what can be concluded about $[D\to E]$ when only $D$ is an FS-domain and $E$ is pointed and continuous?

The paper is organized as follows. Section~2 collects the basic definitions and preliminary facts on domains, topological spaces, and function spaces. Section~3 develops the finite-refinement and palette methods, proves the characterization of spaces preserving RB-domains, and relates it to RW$^*$-spaces. Section~4 constructs the countably based compact, core compact, core coherent web space that is not a $C$-space, thereby answering Ern\'e's question. Section~5 proves the retraction theorem for Lawson-compact domains and derives its applications to solve the problem preposed by Shen et al.

\section{Preliminaries}\label{sec:prelim}

In this section, we recall the preliminaries of domain theory and fix our notation. For domain theory, we refer the reader to \cite{AbramskyJung1994,CLD}; for general non-Hausdorff topology, the reader is directed to \cite{Gou13}.

Write $\N=\{1,2,\ldots\}$ and $\Nzero=\{0,1,2,\ldots\}$.
A directed set is nonempty and every two of its elements have an upper
bound in the set. A dcpo is a poset in which every directed set has a
supremum. For $a,x$ in a dcpo $D$, write $a\ll x$ if every directed
$E\subseteq D$ with $x\leq\sup E$ contains $e$ with $a\leq e$.
Put $\Down x=\{a:a\ll x\}$ and $\Up a=\{x:a\ll x\}$.
A \emph{domain} is a continuous dcpo: $\Down x$ is directed with
supremum $x$ for every $x$. A domain which is also a complete lattice is called a continuous lattice. A basis $B$ is a subset of $D$
such that $B\cap\Down x$ is directed with supremum $x$ for every $x\in D$. An element $k$ is compact
if $k\ll k$; an algebraic domain has its set $\Komp(D)$ of compact elements as a basis. We assume domains nonempty; the empty topological space is allowed.

The Scott-topology $\sigma(D)$ consists of all Scott-open sets, i.e., upper sets $U$ with the property that $E$ directed and $\sup E\in U$ imply $U\cap E\ne \varnothing$. Write $\Sc D=(D,\sigma(D))$. The Lawson-topology
$\Law(D)$ is generated by $\sigma(D)$ and the sets $D\setminus\up a$,
where $\up a=\{x:a\leq x\}$. A domain is \emph{Lawson-compact} if
this topology is compact. Compactness does not include a separation
axiom. For continuous lattices the Lawson-topology is compact Hausdorff, the order relation is closed;
we use these standard facts from~\cite[Chapters~III and~VI]{CLD}.

We use the standard interpolation properties of continuous domains
and their bases, including finite interpolation, without further
comment; see~\cite[Chapter~I]{CLD}.

A \emph{deflation} on a dcpo $D$ is a Scott-continuous map
$\delta:D\to D$ with finite image and $\delta\leq\id_D$.
Idempotence is not part of this definition. A directed family of Scott-continuous maps from $D$ to itself with supremum $\id_D$ is an \emph{approximate identity}.
A dcpo is an \emph{RB-domain} if it has an approximate identity of
deflations; it is \emph{bifinite} if these deflations can be chosen
idempotent. We use the general, possibly nonpointed, convention.
Equivalently, RB-domains are Scott-continuous retracts of bifinite
domains; compare~\cite[Theorems~4.1--4.2]{Jung1988} and~\cite{CLD}.

A Scott-continuous map $\alpha:X\to X$ is \emph{finitely separated}
from $\id_X$ if there is a finite set $F\subseteq X$ such that for
every $x\in X$ some $a\in F$ satisfies $\alpha(x)\leq a\leq x$.
An \emph{FS-domain} is a domain with an approximate identity
of finitely separated maps. We use the same convention without a
least-element assumption; in the pointed case this is the standard
definition~\cite[Definition~4.2.9]{AbramskyJung1994}.

We use the standard facts that FS-domains, and in particular
RB-domains, are Lawson-compact, and that Scott-continuous retracts
of RB-domains are RB-domains; see~\cite{AbramskyJung1994,Jung1988,CLD}.

For a space $X$, its specialization preorder is
$x\leq_X y$ if every open set containing $x$ contains $y$.
The open set lattice $\OO(X)$ is ordered by inclusion; joins are unions and
arbitrary meets are interiors of intersections. A space is
\emph{core compact} if $\OO(X)$ is continuous and \emph{core coherent} if
\begin{equation}\label{eq:core-coherent}
 U_1\ll V_1,\quad U_2\ll V_2
 \quad\Longrightarrow\quad U_1\cap U_2\ll V_1\cap V_2.
\end{equation}
It is compact precisely when $X\ll X$ in $\OO(X)$.
A subset is saturated if it is an upper set for the specialization
preorder. A space is coherent if intersections of two compact saturated
sets are compact. Local compactness means that, for each $x\in U$ open,
there is a compact saturated $K$ with $x\in\intr K\subseteq K\subseteq U$.

\begin{lemma}\label{rb:lem:lattice-calculus}
In a continuous lattice, finite joins preserve $\ll$. If
$u\ll\bigvee_{i=1}^n v_i$, there are $w_i\ll v_i$ such that
$u\ll\bigvee_{i=1}^nw_i$.
\end{lemma}
\begin{proof}
For finite-join stability is well-known. For the second assertion,
the finite joins $\bigvee_iw_i$ with $w_i\ll v_i$ form a directed set
with supremum $\bigvee_iv_i$. 
\end{proof}

For a topological space $X$ and a dcpo $L$, the function space
$[X\to L]$ is consist of all continuous functions and carries the pointwise order. Directed suprema are computed
pointwise, so this function space is a dcpo.

Define $\mathbf2=\{0<1\}$ and $\Lambda=\{a,b,\top\}$, with
$a,b$ incomparable and below $\top$. The third test poset is
$\Test=\{\bot,p,q,r,s,\top\}$, ordered by
$\bot<p,q<r,s<\top$, with $p\parallel q$ and $r\parallel s$.

\begin{figure}[ht]
\centering
\begin{tikzpicture}[scale=.8,every node/.style={font=\small}]
\begin{scope}[xshift=-4cm]
\draw (0,0)--(0,2);
\fill (0,0) circle(2pt) node[left]{$0$};
\fill (0,2) circle(2pt) node[left]{$1$};
\node at (0,-.6){$\mathbf2$};
\end{scope}
\begin{scope}
\draw (-.8,0)--(0,2)--(.8,0);
\fill (-.8,0) circle(2pt) node[left]{$a$};
\fill (.8,0) circle(2pt) node[right]{$b$};
\fill (0,2) circle(2pt) node[above]{$\top$};
\node at (0,-.6){$\Lambda$};
\end{scope}
\begin{scope}[xshift=4cm]
\draw (0,0)--(-.8,1)--(-.8,2)--(0,3)--(.8,2)--(.8,1)--(0,0);
\draw (-.8,1)--(.8,2) (.8,1)--(-.8,2);
\fill (0,0) circle(2pt) node[below]{$\bot$};
\fill (-.8,1) circle(2pt) node[left]{$p$};
\fill (.8,1) circle(2pt) node[right]{$q$};
\fill (-.8,2) circle(2pt) node[left]{$r$};
\fill (.8,2) circle(2pt) node[right]{$s$};
\fill (0,3) circle(2pt) node[above]{$\top$};
\node at (0,-.8){$\Test$};
\end{scope}
\end{tikzpicture}
\caption{The three typical finite posets.}
\label{fig:tests}
\end{figure}

Every finite poset is bifinite. If $P$ is a dcpo with a least element and $U$ open in $X$, write
$U\searrow a$ for the map with value $a$ on the open set $U$ and value $\bot$ elsewhere. The map $U\searrow a$ is continuous. For a compact $a\in P$, $U\ll V$ implies
$U\searrow a\ll V\searrow a$: apply $U\ll V$ to the directed
family of inverse images of the Scott-open set $\up a$.

\section{Topological spaces preserving RB-domains}\label{sec:characterization}

In this section, we study those topological spaces X such that $[X\to L]$ is an RB-domain whenever L is an RB-domain. We find that such topological spaces are closely related to the web spaces introduced by Ern\'e.

\subsection{Web spaces and finite refinements}\label{sec:web-refinements}

A subset $W$ of a space $X$ is a \emph{web around $x$} if $x\in W$
and every $y\in W$ has a common lower bound with $x$ lying in $W$.
The space is a \emph{web space} if, whenever $x\in U$ with $U$ open,
there is such a web with $x\in\intr W\subseteq W\subseteq U$.
This is Ern\'e's terminology~\cite{Erne2009,Erne2019}.
A \emph{coframe} is a complete lattice satisfying
$a\vee\bigwedge_i b_i=\bigwedge_i(a\vee b_i)$ for every family $(b_i)$.

We also use Ern\'e's cardinal convention. For a cardinal $\kappa\geq2$,
let $\PP_\kappa L=\{A\subseteq L:|A|<\kappa\}$, and put
$\PP_\infty L=\PP L$. In particular, $\PP_\omega L$ consists of all finite subsets of $L$,
including the empty set. In a complete lattice define
\begin{equation}\label{eq:kappa-approx}
 \DD_\kappa b=\bigcap\{\dn A:A\in\PP_\kappa L,\ b\leq\bigvee A\}.
\end{equation}
The lattice is \emph{$\kappa$-distributive} if
$b=\bigvee\DD_\kappa b$ for every $b$.
The empty cover is included, so $\DD_\kappa0=\varnothing$.
For $b\ne0$, membership $a\in\DD_3b$ says that
$b\leq c\vee d$ implies $a\leq c$ or $a\leq d$.


\begin{lemma}\label{lem:web-coframe}
A space $X$ is a web space if and only if $\OO(X)$ is a coframe.
\end{lemma}
\begin{proof}
We include the argument from~\cite{Erne2009} to specify the topological
condition precisely. Suppose $\OO(X)$ is a coframe, so the lattice of
closed sets is a frame. For $x\in U$ open put
$W=\up_X(U\cap\dn_Xx)$. This is a web around $x$ contained in $U$.
If $x\notin\intr W$, then $x\in\Cl(X\setminus W)$, and the frame
identity for closed sets gives
\[
 \dn_Xx\cap\Cl(X\setminus W)
 =\dn_Xx\cap\Cl(\bigcup_{y\in X\setminus W}\dn_Xy)=\Cl\!\left(\bigcup_{y\notin W}(\dn_Xx\cap\dn_Xy)\right).
\]
Every common lower bound on the right lies outside $U$, by the
definition of $W$. Thus the right side is contained in the closed set
$X\setminus U$, contradicting $x\in U$. Hence $W$ is a neighborhood.

Conversely, let $X$ be a web space, and let $F,F_i$ be closed.
If $x\in F\cap\Cl(\bigcup_iF_i)$, every open neighborhood $U$ of
$x$ contains a web neighborhood $W$ of $x$. Some
$y\in\intr W\cap F_i$ exists. A common lower bound $z\in W$ of
$x,y$ lies in $U\cap F\cap F_i$, since closed sets are lower.
It follows that
$F\cap\Cl(\bigcup_iF_i)\subseteq\Cl(\bigcup_i(F\cap F_i))$.
The reverse inclusion always holds, proving the frame identity for
closed sets and the coframe identity for opens.
\end{proof}
	
For core compact spaces, the coframe condition can also be expressed
in terms of the Lawson-topology. The following equivalence relates it
to the $\kappa$-distributive laws used in \namedref{Section}{sec:example}.
	
\begin{proposition}\label{prop:coframe-lawson}
For a continuous distributive lattice $L$, the following are equivalent:
\begin{enumerate}
\item $L$ is a coframe;
\item binary join is continuous from $(L,\Law(L))^2$ to $(L,\Law(L))$;
\item $L$ is $3$-distributive;
\item $L$ is $\kappa$-distributive for every finite $\kappa\geq2$.
\end{enumerate}
\end{proposition}
\begin{proof}
Assume (i). For fixed $a\in L$, the map $j_a(b)=a\vee b$ preserves
arbitrary meets. Its left adjoint is
$\ell_a(c)=\bigwedge\{b:c\leq a\vee b\}$, so
$j_a^{-1}(\up c)=\up\ell_a(c)$. Moreover, $j_a$ preserves directed
joins. Consequently $j_a$ is Lawson-continuous. Commutativity gives
separate Lawson-continuity of join. The Lawson-space of $L$ is compact
Hausdorff, so Lawson's joint-continuity theorem for compact Hausdorff
semitopological semilattices~\cite[Proposition 7.4]{Lawson1974} gives (ii).
This use of the joint-continuity theorem is also recorded explicitly
in~\cite[Section~7]{BanakhBardyla}.

Assume (ii), fix $n\geq1$, and take $y\ne0$ and $u\ll y$.
The set $\Cov_n(y)=\{(a_1,\ldots,a_n):y\leq\bigvee_j a_j\}$
is compact in $(L,\Law(L))^n$. For every tuple in this set,
distributivity gives $y=\bigvee_j(y\wedge a_j)$.
\namedref{Lemma}{rb:lem:lattice-calculus} supplies $p_j\ll y\wedge a_j$
with $u\leq\bigvee_jp_j$. The rectangles $\prod_j\Up p_j$ so
obtained cover $\Cov_n(y)$. Choose a finite subcover, with rows
$(p_{ij})_{j=1}^n$ for $1\leq i\leq m$. Finite distributivity gives
\begin{equation}\label{eq:finite-expansion}
 u\leq\bigwedge_{i=1}^m\bigvee_{j=1}^np_{ij}
   =\bigvee_{\varepsilon\in\{1,\ldots,n\}^{\{1,\ldots,m\}}}
                    \bigwedge_{i=1}^mp_{i,\varepsilon(i)}.
\end{equation}
Each meet on the right belongs to $\DD_{n+1}y$: any nonempty cover
of $y$ by at most $n$ elements can be padded by repetitions, belongs
to one chosen rectangle, and therefore dominates the selected entry
in that row in one coordinate. Taking all $u\ll y$ proves
$y=\bigvee\DD_{n+1}y$. The equality at $0$ follows from the empty-cover
convention. Thus (iv) holds, and (iv) implies (iii).

Assume (iii), and set $v=\bigwedge_i(a\vee b_i)$.
If $v=0$ or the family is empty, the coframe equality is immediate.
Otherwise each $r\in\DD_3v$ is either below $a$, or, if $r\nleq a$,
below every $b_i$. Hence $r\leq a\vee\bigwedge_i b_i$.
Taking their join proves the nontrivial inequality in the coframe
identity. The opposite inequality holds in every complete lattice.
\end{proof}

For core compact $X$, write $\LUC$ for Lawson- continuity of
$\mu:\OO(X)^2\to\OO(X)$, $\mu(A,B)=A\cup B$.
Say that $X$ has the \emph{finite binary refinement property} $\FBR$
if, for every $U\ll V$, there are finitely many pairs $(P_\nu,Q_\nu)$
of open sets such that
\begin{equation}\label{eq:FBR}
 P_\nu,Q_\nu\ll V,\qquad U\ll P_\nu\cup Q_\nu,
\end{equation}
and every cover $V\subseteq A\cup B$ satisfies
$P_\nu\ll A$, $Q_\nu\ll B$ for some $\nu$.

\begin{theorem}\label{rb:thm:FBR-LUC}\label{thm:web-equivalences}
For core compact $X$, the following are equivalent: $X$ is a web
space; $\OO(X)$ is a coframe; $\OO(X)$ is $3$-distributive;
$X$ satisfies $\LUC$; $X$ satisfies $\FBR$.
\end{theorem}
\begin{proof}
The first four are equivalent by \namedref{Lemma}{lem:web-coframe} and
\namedref{Proposition}{prop:coframe-lawson}. Assume $\LUC$. The set
$\Cov_2(V)=\{(A,B):V\subseteq A\cup B\}$ is compact in the product Lawson-topology.
 For each $(A,B)\in\Cov_2(V)$,
$V=(A\cap V)\cup(B\cap V)$, so there are
$P\ll A\cap V$, $Q\ll B\cap V$ with $U\ll P\cup Q$.
The rectangles $\Up P\times\Up Q$ cover $\Cov_2(V)$ and satisfy
\eqref{eq:FBR}. A finite subcover proves $\FBR$.

Conversely, assume $\FBR$. Binary union is continuous for the product
Scott-topology. Indeed, the sets $P\cup Q$ with $P\ll A$ and
$Q\ll B$ form a directed family with supremum $A\cup B$.
If a Scott-open set $\mathcal W\subseteq\OO(X)$ contains $A\cup B$,
choose such $P,Q$ with $P\cup Q\in\mathcal W$. Then
$\Up P\times\Up Q$ is a neighborhood of $(A,B)$ mapped into
$\mathcal W$. It remains to treat lower subbasic opens. If $V\nsubseteq A\cup B$, choose $U\ll V$ with
$U\nsubseteq A\cup B$ and take the finite pairs in \eqref{eq:FBR}.
For each $\nu$, either $P_\nu\nsubseteq A$ or $Q_\nu\nsubseteq B$.
In the first case use the neighborhood
$(\OO(X)\setminus\up P_\nu)\times\OO(X)$; otherwise use
$\OO(X)\times(\OO(X)\setminus\up Q_\nu)$.
Their finite intersection contains $(A,B)$ and avoids $\Cov_2(V)$,
by the refinement property. Thus $\Cov_2(V)$ is closed, which proves
Lawson-continuity of union.
\end{proof}

The step-function construction will require refinements of finite
covers. We now formulate this version of the finite binary refinement property.

\begin{definition}\label{rb:def:finite-palette}
Let $I$ be a finite nonempty set and let $U\ll V$ in $\OO(X)$.  A
\emph{buffered $I$-refinement} of $(U,V)$ is a family
$\mathcal R=((U_i,V_i))_{i\in I}$
such that
\[
 U\subseteq\bigcup_{i\in I}U_i,
 \qquad
 U_i\ll V_i\ll V\quad(i\in I).
\]
A finite set $\mathscr P$ of buffered $I$-refinements is an
\emph{$I$-palette for $(U,V)$} if, for every family $(A_i)_{i\in I}$ of open
sets with
$V\subseteq\bigcup_{i\in I}A_i,$
there is $((U_i,V_i))_{i\in I}\in\mathscr P$ such that
\[
 V_i\ll A_i\qquad(i\in I).
\]
\end{definition}

\begin{proposition}\label{rb:prop:finite-palettes}
Let $X$ be core compact.  The following are equivalent:
\begin{enumerate}[label=\textup{(\roman*)}]
\item $X$ satisfies $\LUC$;
\item for every finite nonempty set $I$ and every $U\ll V$, there exists a
finite $I$-palette for $(U,V)$.
\end{enumerate}
\end{proposition}

\begin{proof}
Assume (i).  For finite $I$, the map
\[
 \mu_I:\OO(X)^I\to\OO(X),
 \qquad
 \mu_I((A_i)_i)=\bigcup_{i\in I}A_i,
\]
is continuous in the product Lawson-topology by induction on $|I|$.  Therefore
$$
 \Cov_I(V)=\{(A_i)_i:V\subseteq\bigcup_iA_i\}
 =\mu_I^{-1}(\uparrow V)
$$
is compact. Fix $U\ll V$.  For each $(A_i)_i\in\Cov_I(V)$, since
$V=\bigcup_{i\in I}(A_i\cap V),$
\namedref{Lemma}{rb:lem:lattice-calculus} gives $P_i\ll A_i\cap V$ such that
$U\ll \bigcup_iP_i$.  Hence
$P_i\ll V$.
The rectangles $\prod_i\Up P_i$ obtained in this way form an open cover of
$\Cov_I(V)$.  Choose a finite subcover indexed by $\nu=1,\dots,N$, with tuples
$(P_i^\nu)_{i\in I}$.  For each $\nu$, use
\namedref{Lemma}{rb:lem:lattice-calculus} once more to choose
$U_i^\nu\ll P_i^\nu$ such that
$U\subseteq\bigcup_iU_i^\nu.$
Then
$\left\{((U_i^\nu,P_i^\nu))_{i\in I}:1\leq\nu\leq N\right\}$
is the required $I$-palette.

Conversely, take an $\{0,1\}$-palette for $U\ll V$.
For each of its members $((U_i,V_i))_{i=0,1}$, finite-join
stability of $\ll$ gives
\[
 U\subseteq U_0\cup U_1\ll V_0\cup V_1,
 \qquad V_0,V_1\ll V.
\]
Thus the finitely many pairs $(V_0,V_1)$ satisfy $\FBR$, including
the required relation $U\ll V_0\cup V_1$. Their refinement
property is exactly the palette property.
\namedref{Theorem}{rb:thm:FBR-LUC} now gives $\LUC$.
\end{proof}

\subsection{Deflations on function spaces}
The construction below develops the pointwise directed step-function
method of Lawson and Xi~\cite{LawsonXi2015} and of Xi, Xu, and
Lawson~\cite{XiXuLawson2016}. Finite buffered refinements provide uniform choices for all maps in the function space.

Throughout this subsection, assume that $X$ is compact, core compact, core
coherent, and satisfies $\LUC$.  Let $D$ be an arbitrary RB-domain, not
necessarily pointed, and let $\mathcal E$ be an approximate identity of deflations on $D$.

\begin{definition}\label{rb:def:activation-open}
For $a\in D$ and $\delta\in \mathcal{E}$, define
$H_{\delta,a} =\{y\in D:a\leq\delta(y)\}.$
For $f\in[X\to D]$, define
$\Theta_{\delta,a}(f)=f^{-1}(H_{\delta,a})\in\OO(X).$
\end{definition}

\begin{lemma}\label{rb:lem:activation-open}
For every $a\in D$ and $\delta\in \mathcal{E}$, the set $H_{\delta,a}$ is Scott-open.  Moreover, the map
\[
 \Theta_{\delta,a}:[X\to D]\to\OO(X),
 \qquad
 \Theta_{\delta,a}(f)=f^{-1}(H_{\delta,a}),
\]
is Scott-continuous.
\end{lemma}
\begin{proof}
The set $H_{\delta,a}$ is upper because $\delta$ is monotone.  Let
$E\subseteq D$ be directed and suppose $\bigvee E\in H_{\delta,a}$.  Then
$a\leq\delta\left(\bigvee E\right)=\bigvee\delta(E).$
The set $\delta(E)$ is finite and directed, so it has a greatest element
$\delta(e)$ for some $e\in E$.  Hence $a\leq\delta(e)$ and
$e\in H_{\delta,a}$.  Thus $H_{\delta,a}$ is Scott-open.

If $f\leq g$, then $\Theta_{\delta,a}(f)\subseteq \Theta_{\delta,a}(g)$ because
$H_{\delta,a}$ is upper.  If $(f_i)$ is directed with supremum $f$, then
\[
 \Theta_{\delta,a}(f)
 =f^{-1}(H_{\delta,a})
 =\bigcup_i f_i^{-1}(H_{\delta,a})
 =\bigcup_i \Theta_{\delta,a}(f_i).
\]
Hence $\Theta_{\delta,a}$ preserves directed suprema.
\end{proof}


\begin{definition}\label{rb:def:labelled-step}
A \emph{labelled step function} is a triple
$t=(U_t,V_t,a_t)$
with $U_t,V_t\in\OO(X)$, $U_t\ll V_t$, and $a_t\in D$.
We write
\[
 \supp(t)=U_t,
 \qquad
 \guard(t)=V_t,
 \qquad
 \Lab(t)=a_t.
\]
The labelled step function $t$ is \emph{$(\delta,f)$-active} if
$V_t\ll \Theta_{\delta,a_t}(f).$
\end{definition}

\begin{remark}
Intuitively, a labelled step function $t=(U_t,V_t,a_t)$ is a partially defined function from $X$ to $D$. It takes value $a_t$ on $U_t$, and is undefined outside $U_t$ (see Figure \ref{label-step-function}). The labelled step function $t$ is $(\delta,f)$-active if the guard $V_t$ is way-below the inverse image of ${\up a_t}$ under $\delta\circ f$. 

\begin{figure}[htbp]
	\centering
	\begin{tikzpicture}[scale=0.6, transform shape, font=\large]
	
	\fill[
	pattern=north east lines,
	pattern color=black
	] (0,0) rectangle (8.5,7.5);
	
	\fill[white]
	(1.2,7.5)
	-- (7.7,7.5)
	.. controls (7.75,5.4) and (7.3,3.5) ..
	(6.45,3.05)
	.. controls (5.35,2.4) and (4.5,2.05) ..
	(3.2,2.55)
	.. controls (1.25,3.0) and (1.18,5.5) ..
	(1.2,7.5)
	-- cycle;
	
	\draw (0,0) rectangle (8.5,7.5);
	
	\draw
	(1.2,7.5)
	.. controls (1.18,5.5) and (1.25,3.0) ..
	(3.2,2.55)
	.. controls (4.5,2.05) and (5.35,2.4) ..
	(6.45,3.05)
	.. controls (7.3,3.5) and (7.75,5.4) ..
	(7.7,7.5);
	
	\node at (4.3,4.9) {$a_t$};
	
	\node[fill=white, inner sep=2pt]
	at (4.1,0.85) {undefined};
	
	\node at (7.05,3.2) {$U_t$};
	
	\node[right] at (8.8,1.9) {$X$};
	
\end{tikzpicture}
    \caption{Partial function of the labelled step function $t$.}
\label{label-step-function}
\end{figure}
\end{remark}

If $t$ is $(\delta,f)$-active and $x\in U_t$, then
\begin{equation}\label{rb:eq:active-below-f}
 a_t\leq\delta(f(x))\leq f(x).
\end{equation}

For $a,b\in D$, define
$M_\delta(a,b) =\Min\bigl(\delta(D)\cap{\uparrow a}\cap{\uparrow b}\bigr).$ Here ``$\Min$'' means the set of all minimal elements of a poset. 
The set $M_\delta(a,b)$ is finite.  If $y\in \delta(D)$ is an upper bound of $a$ and $b$, then
there is $m\in M_\delta(a,b)$ with $m\leq y$.

\begin{lemma}\label{rb:lem:lift-package}
Let $t,s$ be labelled step functions whose labels are incomparable. Let $\delta\in \mathcal{E}$. 
There exists a finite set $\mathcal L_\delta(t,s)$ of labelled step
functions such that
\begin{enumerate}[label=\textup{(\roman*)}]
\item $a_u\in M_\delta(a_t,a_s)$ for all $u\in\mathcal L_\delta(t,s)$;
\item if $f\in[X\to D]$ makes both $t$ and $s$ $(\delta,f)$-active, then
$$U_s\cap U_t\subseteq \bigcup\{U_u~|~u\in\mathcal L_\delta(t,s),\ u\ \text{is}\ (\delta,f)\text{-active}\}.$$
\end{enumerate}
\end{lemma}
\begin{proof}
If $U_t\cap U_s=\varnothing$, the empty family suffices.
If $M_\delta(a_t,a_s)=\varnothing$ and $U_t\cap U_s\ne\varnothing$, suppose $f$ makes both $t$ and $s$ $(\delta,f)$-active. Then 
a point $x\in U_t\cap U_s$ would give the common upper bound
$\delta(f(x))\in \delta(D)$ of their labels, a contradiction.
Thus the empty family also suffices in that case.

Assume that
\[
 M=M_\delta(a_t,a_s)\neq\varnothing,
 \qquad
 U_t\cap U_s\neq\varnothing.
\]
Core coherence gives
$U_t\cap U_s\ll V_t\cap V_s.$
Choose a finite $M$-palette for this pair by
\Cref{rb:prop:finite-palettes}.  For every buffered refinement
$((W_m,Z_m))_{m\in M}$ occurring in that palette, include the labelled step functions
\[
 (W_m,Z_m,m),
 \qquad m\in M,
\]
in $\mathcal L_\delta(t,s)$. The resulting set is finite, and
every output label belongs to $M$. 

Now suppose $t$ and $s$ are active for $(\delta,f)$.  Then
$V_t\cap V_s \subseteq \Theta_{\delta,a_t}(f)\cap \Theta_{\delta,a_s}(f).$
For $x\in X$, the condition
$x\in \Theta_{\delta,a_t}(f)\cap \Theta_{\delta,a_s}(f)$ means that
$\delta(f(x))\in \delta(D)$ is an upper bound of $a_t,a_s$.  Hence
\begin{equation}\label{rb:eq:activation-cover-by-M}
 \Theta_{\delta,a_t}(f)\cap \Theta_{\delta,a_s}(f)
 \subseteq\bigcup_{m\in M}\Theta_{\delta,m}(f).
\end{equation}
Therefore
$V_t\cap V_s\subseteq\bigcup_{m\in M}\Theta_{\delta,m}(f).$
By the defining property of the palette, one of its buffered refinements
$((W_m,Z_m))_{m\in M}$ satisfies
\[
 Z_m\ll \Theta_{\delta,m}(f)
 \qquad(m\in M).
\]
All corresponding labelled step functions are $(\delta,f)$-active, and their supports cover
$U_t\cap U_s$. It proves (ii).
\end{proof}

Fix some $\delta\in \mathcal{E}$. For every unordered pair of labelled step functions with incomparable labels,
fix one set
$\mathcal L_\delta(t,s)$ as in \Cref{rb:lem:lift-package}. For a pair $t,s$ with $U_t\cap U_s=\varnothing$ or $M_\delta(a_t,a_s)=\varnothing$, we take the empty family. The set $\mathcal L_\delta(t,s)$ is called \emph{a $\delta$-lift of $t$ and $s$}.

\begin{definition}\label{rb:def:delta-closed}
A set $C$ of labelled step functions is \emph{$\delta$-closed under lifts} if
\[
 t,s\in C,
 \quad
 \Lab(t)\parallel\Lab(s)
 \quad\Longrightarrow\quad
 \mathcal L_\delta(t,s)\subseteq C.
\]
\end{definition}

\begin{lemma}\label{rb:lem:finite-lift-closure}
Every finite set $F$ of labelled step functions is contained in a finite set that is
$\delta$-closed under lifts.
\end{lemma}

\begin{proof}
Define inductively
\[
 F_0=F,\qquad
 F_{n+1}=F_n\cup
 \bigcup\{\mathcal L_\delta(t,s):t,s\in F_n,\
                 \Lab(t)\parallel\Lab(s)\}.
\]
Every $F_n$ is finite. By \namedref{Lemma}{rb:lem:lift-package},
each newly added labelled step function has its label in $\delta(D)$.

Suppose a labelled step function first appears at stage $n\geq2$.
At least one of the labelled step functions generating it must appear at stage $n-1$; otherwise both of them belong to $F_{n-2}$, and the function would already belong
to $F_{n-1}$. Repeating this argument we obtain
$t_0,\ldots,t_n$ with
\[
 t_j\in F_j\setminus F_{j-1}\quad(1\leq j\leq n),\qquad
 \Lab(t_0)<\Lab(t_1)<\cdots<\Lab(t_n).
\]
The $n$ labels with indices $1,\ldots,n$ are distinct elements
of $\delta(D)$. Hence $n\leq|\delta(D)|$.
It follows that the increasing sequence $F_0\subseteq F_1\subseteq F_2\subseteq \ldots$ will stabilize, and this finite stable set contains $F$ and is $\delta$-closed under lifts.
\end{proof}

The next construction is the point at which compactness of $X$ is used.

\begin{lemma}\label{rb:lem:guard-system}
For every $\delta\in\mathcal E$, there is a finite set $G_\delta$ of labelled step functions
such that, for every $f\in[X\to D]$,
$$X=\bigcup\{U_u~|~u\in G_\delta,\ u\ \text{is}\ (\delta,f)\text{-active}\}.$$
\end{lemma}
\begin{proof}
Let
$I=\Min(\delta(D)).$
This is a finite nonempty set.  For every $f\in[X\to D]$,
\begin{equation}\label{rb:eq:minimal-cover}
 X=\bigcup_{m\in I}\Theta_{\delta,m}(f).
\end{equation}
Indeed, for $x\in X$, the element $\delta(f(x))\in \delta(D)$ lies above a
minimal element $m\in I$.

Since $X$ is compact, $X\ll X$ in $\OO(X)$.  Choose a finite $I$-palette for
$(X,X)$.  For every buffered refinement $((U_m,V_m))_{m\in I}$ in that palette,
include all labelled step functions $(U_m,V_m,m)$ in $G_\delta$.  The set $G_\delta$ is finite.
For a fixed $f$, apply the palette property to the cover
\Cref{rb:eq:minimal-cover}.  One buffered refinement has
$V_m\ll \Theta_{\delta,m}(f)$ for every $m$, and its supports cover $X$.
\end{proof}

\begin{definition}\label{rb:def:lift-system}
A finite set $C$ of labelled step functions is a \emph{lift system for $\delta$} if
\begin{enumerate}[label=\textup{(\roman*)}]
\item $G_\delta\subseteq C$;
\item $C$ is $\delta$-closed under lifts.
\end{enumerate}
\end{definition}

By \Cref{rb:lem:finite-lift-closure}, every finite set of labelled step functions is contained in a
finite lift system for $\delta$ after adjoining $G_\delta$.

Fix $\delta\in\mathcal E$ and a finite lift system $C$ for $\delta$.
For $f\in[X\to D]$ and $x\in X$, put
$$S_C^\delta(f,x) =\{a_t:t\in C,\ t\text{ is }(\delta,f)\text{-active},\ x\in U_t\}.$$

\begin{lemma}\label{rb:lem:active-labels-directed}
For every $f\in[X\to D]$ and $x\in X$, the set $S_C^\delta(f,x)$ is a finite directed subset of $D$.  It therefore has a greatest element.
\end{lemma}
\begin{proof}
Finiteness is immediate.  Nonemptiness follows from
\Cref{rb:lem:guard-system} and the inclusion $G_\delta\subseteq C$.

Let $a_t,a_s\in S_C^\delta(f,x)$.  If they are comparable, the larger one is an upper
bound belonging to $S_C^\delta(f,x)$.  Suppose they are incomparable.  The labelled step functions
$t,s$ are $(\delta,f)$-active and $x\in U_t\cap U_s$.  By
\Cref{rb:lem:lift-package}, there is
$u\in\mathcal L_\delta(t,s)$ that is $(\delta,f)$-active, contains $x$ in its support, and
has $a_t,a_s\leq a_u$.  Since $C$ is $\delta$-closed under lifts, $u\in C$, and hence $a_u\in S_C^\delta(f,x)$.  Thus $S_C^\delta(f,x)$ is directed. It is well-known that a finite directed subset has a largest element. 
\end{proof}

Define
\begin{equation}\label{rb:eq:def-Phi}
 \Phi_{\delta,C}(f)(x)=\max S_C^\delta(f,x).
\end{equation}

\begin{lemma}\label{rb:lem:Phi-value-continuous}
For every $\delta\in \mathcal{E}$, finite lift system $C$ for $\delta$, and $f\in[X\to D]$, the map $\Phi_{\delta,C}(f):X\to D$ is continuous and
satisfies
$\Phi_{\delta,C}(f)\leq f.$
\end{lemma}

\begin{proof}
Let $O\subseteq D$ be Scott-open.  Since a finite directed set has a greatest
element belonging to the set,
\begin{equation}\label{rb:eq:Phi-preimage}
 \Phi_{\delta,C}(f)^{-1}(O)
 =\bigcup\{U_t:t\in C,\ t\text{ is }(\delta,f)\text{-active},\ a_t\in O\}.
\end{equation}
The right-hand side is open, so $\Phi_{\delta,C}(f)$ is continuous.

Let $x\in X$.  Every $a_t\in S_C^\delta(f,x)$ satisfies
$a_t\leq\delta(f(x))\leq f(x)$ by \Cref{rb:eq:active-below-f}.  The greatest
element of $S_C^\delta(f,x)$ therefore satisfies the same inequality.
\end{proof}

\begin{proposition}\label{rb:prop:Phi-deflation}
The map
$\Phi_{\delta,C}:[X\to D]\longrightarrow[X\to D]$
is a deflation.
\end{proposition}

\begin{proof}
The preceding lemma shows that $\Phi_{\delta,C}$ is well defined and lies below
the identity.

If $f\leq g$, then
$\Theta_{\delta,a}(f)\subseteq \Theta_{\delta,a}(g)$
for every $a$.  Hence every labelled step function active for $(\delta,f)$ is active for $(\delta,g)$.  Therefore
$S_C^\delta(f,x)\subseteq S_C^\delta(g,x)$ and
$\Phi_{\delta,C}(f)(x)\leq\Phi_{\delta,C}(g)(x)$
for every $x$.  Thus $\Phi_{\delta,C}$ is monotone.

Let $(f_i)_{i\in I}$ be directed with supremum $f$.  For a fixed labelled step function
$t=(U_t,V_t,a_t)$, the set of maps for which $t$ is active is
$\Theta_{\delta,a_t}^{-1}(\Up V_t),$
which is Scott-open by \Cref{rb:lem:activation-open}, since
$\Up V_t$ is Scott-open in the continuous lattice $\OO(X)$. Take an arbitrary $x\in X$. There is some $(\delta,f)$-active labelled step function $t\in C$ with $x\in U_t$ such that 
$$a_t=\max S_C^\delta(f,x)=\Phi_{\delta,C}(f)(x).$$  
It follows that $t$ is already active for some $(\delta,f_{i_0})$. Hence 
$$\Phi_{\delta,C}(f)(x)=a_t\le \max S_C^\delta(f_{i_0},x)\le \bigvee_{i\in I}\max S_C^\delta(f_i,x)=(\bigvee_{i\in I}\Phi_{\delta,C}(f_i))(x).$$  
Therefore
$\Phi_{\delta,C}(f)\le  \bigvee_{i\in I}\Phi_{\delta,C}(f_i)$. The reverse inequality follows from the monotonicity. Thus $\Phi_{\delta,C}$ is Scott-continuous.

Let $\mathcal{P}(C)$ be the power set of $C$. Then $\mathcal{P}(C)$ is finite. Define a map $\xi:X\to \mathcal{P}(C)$ sending $x\in X$ to $\{t\in C: x\in U_t\}$. It defines an equivalence relation $\sim$ on $X$ by
$$x\sim y\Longleftrightarrow \xi(x)=\xi(y).$$
The set $X/\sim$ of all equivalent classes under $\sim$ is finite. Let
$$\mathcal{S}:=\{f:X\to D~|~f(X)\subseteq \{a_t:t\in C\}, x\sim y\Rightarrow f(x)=f(y)\}.$$
Then $\mathcal{S}$ is a finite set. It is clear that the image of any function $\Phi_{\delta,C}(f)$ is contained in $\{a_t:t\in C\}$. In addition, if $x\sim y$, then for any $(\delta,f)$-active function $t\in C$,
$$x\in U_t\Longleftrightarrow t\in \xi(x)\Longleftrightarrow t\in \xi(y)\Longleftrightarrow y\in U_t.$$
It implies $S_C^\delta(f,x)=S_C^\delta(f,y)$. Hence
$$\Phi_{\delta,C}(f)(x)=\max S_C^\delta(f,x)=\max S_C^\delta(f,y)=\Phi_{\delta,C}(f)(y).$$
So all images of the function $\Phi_{\delta,C}$ are contained in $\mathcal{S}$. As a consequence, it is finite. 
\end{proof}

We now assemble these deflations into an approximate identity. Let $\mathfrak D$ be the set of all maps $\Phi_{\delta,C}$, where
$\delta\in\mathcal E$ and $C$ is a lift system for $\delta$.

\begin{lemma}\label{rb:lem:family-directed}
The family $\mathfrak D$ is directed under the pointwise order on
$[[X\to D]\to[X\to D]]$.
\end{lemma}
\begin{proof}
The family is nonempty: choose $\delta\in\mathcal E$ and a finite
lift system for $\delta$ containing $G_\delta$.

Let $\Phi_{\delta_1,C_1},\Phi_{\delta_2,C_2}\in\mathfrak D$.  Choose
$\delta\in\mathcal E$ with $\delta_1,\delta_2\leq\delta$.  By
\Cref{rb:lem:finite-lift-closure}, extend
$C_1\cup C_2\cup G_\delta$
to a finite lift system $C$ for $\delta$. If a labelled step function $(U,V,a)\in C_i$ is $(\delta_i,f)$-active, then
$\Theta_{\delta_i,a}(f)\subseteq \Theta_{\delta,a}(f),$
because $\delta_i\leq\delta$.  Hence the labelled step function is also $(\delta,f)$-active.
For every $x$, this yields
$S_{C_i}^{\delta_i}(f,x)\subseteq S_C^{\delta}(f,x)$. Taking greatest elements gives
\[
 \Phi_{\delta_i,C_i}(f)\leq\Phi_{\delta,C}(f)
 \qquad(i=1,2).
\]
Thus $\mathfrak D$ is directed.
\end{proof}

\begin{lemma}\label{rb:lem:family-sup-id}
The directed supremum of $\mathfrak D$ is the identity map on $[X\to D]$.
\end{lemma}
\begin{proof}
Every member of $\mathfrak D$ lies below the identity.  It remains to prove the
reverse inequality pointwise.  Fix $f\in[X\to D]$, $x\in X$, and
$a\ll f(x)$.  Since $\mathcal E$ is an approximate identity, the family
$\{\delta(f(x)):\delta\in\mathcal E\}$ is directed with supremum $f(x)$.
Choose $\delta\in\mathcal E$ such that
$a\leq\delta(f(x)).$
Then $x\in \Theta_{\delta,a}(f)$.  Core compactness and interpolation in
$\OO(X)$ yield open sets $U,V$ such that
$x\in U\ll V\ll \Theta_{\delta,a}(f).$
Let $t=(U,V,a)$ and extend $\{t\}\cup G_\delta$ to a finite
lift system $C$ for $\delta$.  The labelled step function $t$ is active for $(\delta,f)$ and contains $x$
in its support.  Hence
$a\leq\Phi_{\delta,C}(f)(x).$
Since $D$ is continuous and $f(x)=\bigvee\Down f(x)$, it follows that
$f(x)\leq (\bigvee_{\Phi\in\mathfrak D}\Phi)(f)(x)$. Therefore the supremum of
$\mathfrak D$ is the identity.
\end{proof}

\begin{theorem}\label{rb:thm:sufficiency}
Let $X$ be compact, core compact, and core coherent, and suppose that $X$
satisfies $\LUC$.  Then $[X\to D]$ is an RB-domain for every RB-domain $D$.
\end{theorem}
\begin{proof}
By \Cref{rb:prop:Phi-deflation}, every member of $\mathfrak D$ is a deflation of $[X\to D]$.  By \Cref{rb:lem:family-directed}, the family is
directed, and by \Cref{rb:lem:family-sup-id}, its supremum is the identity.
Therefore $[X\to D]$ is an RB-domain.
\end{proof}

\begin{remark}\label{rb:rem:nonpointed-role}
No least element of $D$ was used.  In a pointed domain, the bottom-valued
constant function can serve as a permanent default approximant.  In the
nonpointed case, \Cref{rb:lem:guard-system} replaces that default by finitely many
minimal values in the finite image $\delta(D)$.  Compactness of $X$, equivalently
$X\ll X$ in $\OO(X)$, is exactly what turns the resulting open cover into a
finite buffered guard palette. This approach has been used in proving Theorem 5.6 of \cite{LuanLi2025}.
\end{remark}

\subsection{The equivalence theorem}

In this subsection, we prove the main result which totally characterizes those spaces preserving RB-domains. The following lemma is Theorem 3.1 of  \cite{LuanLiSG}.

\begin{lemma}\label{lem:classical-tests}
Let $X$ be a topological space. The function spaces $[X\to\mathbf2]$, $[X\to\Lambda]$, and $[X\to\Test]$ are
continuous domains iff $X$ is compact, core compact, and core coherent.
\end{lemma}




\begin{lemma}\label{lem:coordinate-waybelow}
If $g\ll f$ in $[X\to\Test]$, then
$g^{-1}(\up r)\ll f^{-1}(\up r)$ and
$g^{-1}(\up s)\ll f^{-1}(\up s)$ in $\OO(X)$.
\end{lemma}
\begin{proof}
If $(O_i)$ is directed and covers $f^{-1}(\up r)$, the maps with
value $\top$ on $O_i$ and $s$ elsewhere have supremum above $f$:
every value outside $\up r$ is below $s$. Some such map dominates
$g$, giving $g^{-1}(\up r)\subseteq O_i$. The other assertion is symmetric.
\end{proof}

These coordinate estimates allow a finite Lawson-compactness argument
in $[X\to\Test]$ to produce finite binary refinements in $\OO(X)$.

\begin{proposition}\label{prop:lawson-test}
Let $X$ be core compact. If $[X\to\Test]$ is a Lawson-compact
domain, then $X$ is a web space.
\end{proposition}
\begin{proof}
It suffices to prove $\FBR$. Fix $U\ll V$, choose $U\ll U'\ll V$,
and put $E=[X\to\Test]$. The set
$K=\up_E(V\searrow p)\cap\up_E(V\searrow q)$
is nonempty and Scott-compact. For each $f\in K$, both
$U'\searrow p$ and $U'\searrow q$ are way-below $f$.
Choose $g_f\ll f$ above these two maps. The Scott-open sets
$\Up_E g_f$ cover $K$, so finitely many $g_1,\ldots,g_m$ suffice.
Write $R_i=g_i^{-1}(\up r)$ and $S_i=g_i^{-1}(\up s)$.
Then $U'\subseteq R_i\cup S_i$ for every $i$.

Retain only indices for which $R_i\cup S_i\subseteq V$.
For each retained index choose $P_i\ll R_i$, $Q_i\ll S_i$ with
$U\ll P_i\cup Q_i$, using interpolation and
\namedref{Lemma}{rb:lem:lattice-calculus}. These choices satisfy
$P_i,Q_i\ll V$.
Given $V\subseteq A\cup B$, set $A_0=A\cap V$, $B_0=B\cap V$,
and define $f$ to have values $\bot$ off $V$, $r$ on $A_0\setminus B_0$,
$s$ on $B_0\setminus A_0$, and $\top$ on $A_0\cap B_0$.
This is a continuous member of $K$, so $g_i\ll f$ for one of the
chosen indices. Since $g_i\leq f$, this index was retained.
\namedref{Lemma}{lem:coordinate-waybelow} gives $R_i\ll A_0$, $S_i\ll B_0$,
and hence $P_i\ll A$, $Q_i\ll B$.
The retained finite family proves $\FBR$.
\end{proof}

Recall the $T_0$-reflection (see \cite{FPR21}) and sobrification (see \cite{KL09}) of a space $X$. Let $X$ be a topological space. The specialization order on $X$ is denoted by $\le$. Define a relation $\sim$ by $x\sim y\Longleftrightarrow x\le y\ \text{and}\ y\le x$. Then $\sim$ is an equivalent relation of $X$. Denote by $X/\sim$ the set of all equivalent classes of $X$ under the relation $\sim$, equipped with the quotient topology. Let $\epsilon_X:X\to X/\sim$ be the natural epimorphism. Then $\epsilon_X$ is the $T_0$-reflection of $X$, that is, $X/\sim$ is a $T_0$ space, and for each $T_0$ space $Y$ and continuous map $f:X\to Y$, there is a unique continuous map ${f/\sim}:{X/\sim} \to Y$ such that $(f/\sim) \circ \epsilon_X=f$. In addition, mappings $(U\in \mathcal{O}(X))\mapsto \epsilon_X(U)$ and $(V\in\mathcal{O}(X/\sim))\mapsto\epsilon_X^{-1}(V) $ give an order isomorphism $\mathcal{O}(X)\cong \mathcal{O}(X/\sim)$. If $X$ is a $T_0$ space, let $X^s$ be the set of all irreducible closed subset of $X$. For open set $U$ of $X$, let $\diamondsuit U:=\{F\in X^s~|~F\cap U\ne \emptyset\}$. All such $\diamondsuit U$ form a sober topology on $X^s$. Define $\eta_X:X\to X^s$ sending $x$ to ${\dn x}$. For each sober space $Y$ and continuous map $f:X\to Y$, there is a unique continuous map $f^s:X^s\to Y$ such that $f^s\circ \eta_X=f$. The map $\eta_X$ is called the sobrification of $X$. In addition, the mappings $(U\in \mathcal{O}(X))\mapsto \diamondsuit U$ and $(V\in \mathcal{O}(X^s))\mapsto \eta_X^{-1}(V)$ give an order isomorphism $\mathcal{O}(X)\cong \mathcal{O}(X^s)$. Consider successively performing the $T_0$-reflection and sobrification on a topological space $X$, as shown in the figure below.

\begin{figure}[htbp]
	\centering
	\[
	\begin{tikzcd}
		X\arrow{r}{\epsilon_X}\arrow[swap]{ddrr}{f} & X/\sim\arrow[dashed]{ddr}{\exists!\,f/\sim} \arrow{r}{\eta_{X/\sim}} 
		& (X/\sim)^s \arrow[dashed]{dd}{\exists!\,(f/\sim)^s} \\
		&& \\
		& & Y
	\end{tikzcd}
	\]
	\caption{Sobrification of the $T_0$-reflection of $X$.}
	\label{fig:t0_sober_diag}
\end{figure}

Suppose $Y$ is a sober space. Define $\xi:[X\to Y]\to [{X/\sim} \to Y]$ sending $f\in [X\to Y]$ to the unque continuous map ${f/\sim}:{X/\sim}\to Y$ with $f={f/\sim}\circ \epsilon_X$. Define $\theta:[{X/\sim}\to Y]\to [{(X/\sim)}^s \to Y]$ sending $g\in [{X/\sim}\to Y]$ to the unque continuous map $g^s:{(X/\sim)}^s\to Y$ with $g=g^s\circ \eta_{X/\sim}$. Then $\xi$ and $\theta$ are bijections. If $Y$ is a sober dcpo, then we obtain a chain of order-isomorphisms:
\begin{align}\label{equ-1}
	[X\to Y]\xrightarrow{\xi} [{X/\sim} \to Y]\xrightarrow{\theta}[{(X/\sim)}^s \to Y].
\end{align}

Furthermore, if $X$ is a compact, core compact and core coherent space, from 
$$\mathcal{O}(X)\cong \mathcal{O}(X/\sim)\cong \mathcal{O}((X/\sim)^s)$$ 
we get $(X/\sim)^s$ is a compact, locally compact, coherent and sober space, i.e., a stably compact space \cite{Gou13}. 

We arrive at the main result of this section.

\begin{theorem}\label{thm:intro-main}
		For a topological space $X$, the following are equivalent:
		\begin{enumerate}
			\item $X$ is compact, core compact, core coherent, and a web space;
			\item The sobrification of the $T_0$-reflection of $X$ is a stably compact web space;
			\item $[X\to L]$ is an RB-domain for every RB-domain $L$;
			\item $[X\to L]$ is a Lawson-compact domain for every RB-domain $L$;
			\item $[X\to\mathbf2]$ and $[X\to\Lambda]$ are continuous domains,
			and $[X\to\Test]$ is a Lawson-compact domain.
		\end{enumerate}
	\end{theorem}
\begin{proof}
The equivalence ``(i)$\Leftrightarrow$(ii)'' follows from the discussion above. Condition (i) implies (iii) by \namedref{Theorem}{thm:web-equivalences} and
\namedref{Theorem}{rb:thm:sufficiency}. Since every RB-domain is Lawson-compact, (iii) implies (iv). Since the three test posets are RB-domains,
(iv) implies (v). Finally, (v) gives compactness, core compactness,
and core coherence by \namedref{Lemma}{lem:classical-tests}, and the web-space
condition by \namedref{Proposition}{prop:lawson-test}.
\end{proof}

\subsection{A glance at Lawson-Mislove's problem}\label{subsec:lm-problem}

In this subsection, we use the theorem obtained in the previous subsection to tackle the Lawson–Mislove's problem.

A domain $L$ is an \emph{L-domain} if every principal ideal $\dn x$ is a complete lattice, and an \emph{$sL$-domain} if every principal ideal is a join semilattice~\cite{Xu2003}. Binary joins in $\dn x$ are computed relative to that principal ideal.

The following definitions are from \cite{KouLuo2003}.

\begin{definition}
	Let $(X,\OO (X))$ be a topological space and
	let $V\in\OO (X)$.
	\begin{enumerate}[label=\textup{(\arabic*)}]
		\item A family of open sets $\{U_i:i\in I\}$ is
		called a \emph{decomposition} of $V$ if
		\begin{enumerate}[label=\textup{(\alph*)}]
			\item $\displaystyle\bigcup_{i\in I}U_i=V$;
			\item $U_i\cap U_j=\varnothing$ whenever
			$i\neq j$ $(i,j\in I)$;
			\item $U_i\neq\varnothing$ for each $i\in I$.
		\end{enumerate}
		\item A decomposition $\{U_i:i\in I\}$ of $V$
		is said to be \emph{maximal} if for each
		decomposition $\{V_j:j\in J\}$ of $V$ and
		for each $j\in J$, there is $i\in I$ such
		that $U_i\subseteq V_j$.
		
		\item $X$ is said to have \emph{property W}
		if each open set has a finite
		maximal decomposition. If $X$ is also
		core compact, then we call it a
		\emph{W-space}.
	\end{enumerate}
\end{definition}

\begin{definition}
	Let $U,V\in\OO (X)$ with $V\subseteq U$.
	A decomposition $\{V_i:i\in I\}$ of $V$
	is said to be \emph{relatively maximal}
	to $U$, denoted by
	$\displaystyle\sum_U V$, if for each
	decomposition $\{U_j:j\in J\}$ of $U$
	and for each $j\in J$,
	$U_j\cap V\neq\varnothing$ implies that
	\[
	U_j\cap V
	=
	\bigcup\bigl\{
	V_i:i\in I,\;
	U_j\cap V_i\neq\varnothing
	\bigr\}.
	\]
\end{definition}

\begin{definition}
	A topological space $X$ is said to have
	\emph{property RW} if for finitely many
	pairs
	\[
	V_1\ll U_1,\;
	V_2\ll U_2,\;
	\ldots,\;
	V_k\ll U_k
	\]
	in $\OO (X)$, $\displaystyle\bigcap_{i=1}^{k}V_i$
	has a finite decomposition which is
	relatively maximal to
	$\displaystyle\bigcap_{i=1}^{k}U_i$.
	A core compact space is called an \emph{RW-space} if it has property RW.
\end{definition}

RW-spaces have the following important property \cite[Theorem 3.6]{KouLuo2003}.

\begin{theorem}\label{rw:eq:Kou-Luo}
	Let $X$ be a topological space. Then $X$ is a RW-space if and only if $[X\to L]$
	is a Lawson-compact domain for every pointed Lawson-compact L-domain $L$.
\end{theorem}

Let $X$ be a topological space. Recall that a \emph{core}  \cite{Erne2009,Erne2019} is a principal filter in the specialization order. A subset $C$ of the space is \emph{supercompact} if any cover by open subsets has a member that contains $C$. Then cores are just supercompact saturated subsets, and open cores are just supercompact open sets.  

\begin{definition}\label{rw:def:RW-star}
	An \emph{RW$^*$-space} is an RW-space $X$ admitting a finite cover
	$X=C_1\cup\cdots\cup C_r$ by open cores.
\end{definition}

We shall use the following elementary observation of Scott-compact domains: Letting $D$ be a Scott-compact domain, then $\Min D$ is a finite subset of compact elements and $D={\up \Min D}$.

\begin{lemma}\label{rw:lem:open-covers}
	Let $X$ be core compact and, for $n\geq1$, put
	\[
	\mathcal C_n(X)=
	\left\{(U_1,\ldots,U_n)\in\OO(X)^n:
	\bigcup_{i=1}^nU_i=X\right\},
	\]
	ordered coordinatewise by inclusion. The following are equivalent:
	\begin{enumerate}
		\item $X$ has a finite cover by open cores;
		\item $\mathcal C_n(X)$ is a Scott-compact domain for every $n\geq1$.
	\end{enumerate}
\end{lemma}
\begin{proof}
	The assertion is immediate for $X=\varnothing$. Assume $X\ne\varnothing$.
	
	(i)$\Rightarrow$(ii). Write $X=\bigcup_{i=1}^rC_i$ with each
	$C_i$ an open core. In particular $X$ is compact, so
	$\mathcal C_n(X)$ is Scott-open in the continuous lattice $\OO(X)^n$:
	it is the inverse image of the Scott-open singleton $\{X\}$ under
	the union map. Hence it is a domain.
	For each map $q:\{1,\ldots,r\}\to\{1,\ldots,n\}$, the tuple
	\[
	\left(\bigcup_{q(i)=1}C_i,\ldots,
	\bigcup_{q(i)=n}C_i\right)
	\]
	belongs to $\mathcal C_n(X)$. Supercompactness of the $C_i$ shows that $\mathcal C_n(X)$ is generated as an upper set by these finitely many tuples. Their principal filters cover $\mathcal C_n(X)$, proving Scott-compactness.
	
	(ii)$\Rightarrow$(i). First, $X$ is compact. The pair $(X,\varnothing)$ is
	a minimal element of the domain $\mathcal C_2(X)$ and hence compact.
	If $(V_\alpha)$ is a directed family of open sets with union $X$, then
	\[
	\sup_\alpha(V_\alpha,X)=(X,X)\geq(X,\varnothing).
	\]
	Compactness of $(X,\varnothing)$ gives $V_\alpha=X$ for some $\alpha$.
	
	Each $\mathcal C_n(X)$ has finitely
	many minimal elements and they generates $\mathcal C_n(X)$ as an upper set. Every coordinate $U_i$ of a minimal element $(U_1,\ldots,U_n)$ of $\mathcal C_n(X)$ is compact. Indeed, replace $U_i$ by
	an arbitrary cover by open subsets of $U_i$ and use compactness of
	$X$ to retain finitely many of them together with the other $U_j$.
	Minimality forces the union of the retained subsets to equal $U_i$.
	
	Let $\mathscr A$ be the finite set of coordinates of the minimal
	elements of $\mathcal C_2(X)$. Every coordinate of every minimal
	elements of $\mathcal C_n(X)$ belongs to $\mathscr A$. To see this, fix a minimal
	element $(U_1,\ldots,U_n)$ and an index $i$. The pair
	\[
	\left(U_i,\bigcup_{j\ne i}U_j\right)
	\]
	is minimal in $\mathcal C_2(X)$. If a cover $(A,B)$ lies below this
	pair, then the tuple with $i$th coordinate $A$ and other coordinates
	$U_j\cap B$ is a cover below the original tuple. Its minimality
	gives $A=U_i$ and $B=\bigcup_{j\ne i}U_j$.
	
	Put
	$$\mathcal{B}=\{(U_1,\ldots,U_n)\in \Min \mathcal{C}_n(X): n\in \mathbb{N}, U_i\ne \varnothing\text{ for each }i\}.$$
	Since nonempty coordinates of a minimal element are pairwise incomparable,
	so 
	$$\{n\in \mathbb{N}: (U_1,\ldots,U_n)\in \mathcal{B}\}$$
	is at most $|\mathscr A|$. Choose a 
	$(C_1,\ldots,C_r)\in \mathcal{B}$ with the largest possible number of
	coordinates. If some $C_i$ were
	not supercompact, its compactness would give
	\[
	C_i=V_1\cup\cdots\cup V_k,
	\qquad V_j\in\OO(X),\quad \varnothing\ne V_j\subsetneq C_i.
	\]
	Choose a minimal element
	$$(C_1,\ldots,C_{i-1},O_1,\ldots,O_k,C_{i+1},\ldots,C_r)$$
	below
	$(C_1,\ldots,C_{i-1},V_1,\ldots,V_k,C_{i+1},\ldots,C_r)$.
	Then $$(C_1,\ldots,C_{i-1},\bigcup_{j=1}^k O_j,C_{i+1},\ldots,C_r)$$ 
	lies below $(C_1,\ldots,C_r)$ and still covers $X$;
	hence it equals that tuple. All other coordinates are unchanged,
	and the $k$ new coordinates have union $C_i$. At least two of them
	are nonempty, since each is a proper subset of $C_i$.
	After deleting empty coordinates, this gives a minimal element in $\mathcal{B}$ with more than $r$ coordinates, a contradiction. Thus every $C_i$ is supercompact.
\end{proof}

The following result is parallel to \cite[Theorem 3.6]{KouLuo2003}.

\begin{theorem}\label{rw:thm:characterization}
	For a topological space $X$, the following are equivalent:
	\begin{enumerate}
		\item $X$ is an RW$^*$-space;
		\item $[X\to L]$ is a Lawson-compact domain for every Lawson-compact $sL$-domain $L$.
	\end{enumerate}
\end{theorem}
\begin{proof}
	The case $X=\varnothing$ is trivial, so assume $X\ne\varnothing$.
	
	(ii)$\Rightarrow$(i). Applying~\namedref{Theorem}{rw:eq:Kou-Luo} shows that $X$ is an RW-space.
	For each $n\geq1$, let
	\[
	P_n=\bigl(\mathcal{P}({\{1,\ldots,n\}})\setminus\{\varnothing\},\subseteq\bigr).
	\]
	This finite poset is a bifinite $sL$-domain; in particular $P_2\cong\Lambda$.
	There is an order isomorphism
	\[
	[X\to P_n]\longrightarrow\mathcal C_n(X),\qquad
	f\longmapsto\bigl(f^{-1}(\up\{i\})\bigr)_{i=1}^n,
	\]
	whose inverse sends $(U_i)_i$ to $x\mapsto\{i:x\in U_i\}$.
	Condition (ii) therefore implies condition (ii) of
	\namedref{Lemma}{rw:lem:open-covers}, which supplies a finite
	open core cover of $X$.
	
	(i)$\Rightarrow$(ii). Fix such a cover
	$X=C_1\cup\cdots\cup C_r$ and a Lawson-compact $sL$-domain $L$.
	Adjoin a new least element $\bot$ and write
	$L_\bot=L\sqcup\{\bot\}$.
	Each principal ideal of $L_\bot$ has a least element and binary joins;
	arbitrary nonempty joins are directed suprema of finite nonempty
	joins. Thus $L_\bot$ is an L-domain (see also \cite[Observation 2.2]{LawsonXu2003'}). By \namedref{Lemma}{lem:top}, $L_\bot$ is also a pointed Lawson-compact domain. By~\namedref{Theorem}{rw:eq:Kou-Luo},
	$F=[X\to L_\bot]$ is a Lawson-compact domain.
	The Scott-topology of $L$ is the subspace topology inherited from $L_\bot$, so $[X\to L]$ is identified with the subposet of maps in $F$ whose values all belong to $L$.
	
	Write
	$\Min L=\{m_1,\ldots,m_s\}$, where each $m_j$ is compact and
	$L=\bigcup_j\up m_j$. Using the step-function notation of
	\namedref{Section}{sec:prelim}, put
	\[
	e_{ij}=C_i\searrow m_j\in F
	\qquad(1\leq i\leq r,\ 1\leq j\leq s).
	\]
	Each $e_{ij}$ is compact in $F$. Indeed, if a directed family
	$(f_\alpha)$ has supremum above $e_{ij}$, compactness of $m_j$ gives
	$C_i\subseteq\bigcup_\alpha f_\alpha^{-1}(\up m_j)$.
	Supercompactness of $C_i$  give
	$C_i\subseteq f_\alpha^{-1}(\up m_j)$ for one $\alpha$, hence
	$e_{ij}\leq f_\alpha$.
	
	Supercompactness of the $C_i$ gives
	\begin{equation}\label{rw:eq:total-maps}
		[X\to L]=\bigcap_{i=1}^r\bigcup_{j=1}^s\up_F e_{ij}.
	\end{equation}
	Indeed, for $f\in F$, the condition $f(C_i)\subseteq L$ is equivalent
	to $C_i\subseteq\bigcup_jf^{-1}(\up m_j)$, and hence to
	$C_i\subseteq f^{-1}(\up m_j)$ for some $j$.
	Each principal filter $\up_F e_{ij}$ is Scott-open and Lawson-closed.
	Consequently~\eqref{rw:eq:total-maps} is a Scott-open, Lawson-closed
	subset of $F$. A Scott-open subdcpo of a continuous dcpo is continuous
	and has the subspace Scott-topology. Its lower subbasic open sets
	are restrictions of lower subbasic open sets of the ambient dcpo,
	so its intrinsic Lawson-topology is no finer than the inherited
	topology. Lawson-compactness of $F$ therefore proves that
	$[X\to L]$ is a Lawson-compact domain.
\end{proof}

It is known that RB-domains has property $m$. The following result is an immediate consequence of \cite[Proposition 4.22]{Jung1988}.

\begin{lemma}\label{lem:retract-six}
	The poset $\mathbb{X}_\bot^\top$ is a Scott-continuous retract of any pointed RB-domain $L$ that is not an $sL$-domain. 
\end{lemma}

\begin{corollary}\label{cor:xmark}
	If $X$ is an RW$^*$-space and $L$ is a pointed RB-domain which is not an $sL$-domain, in addition, the function space $[X\to L]$ is a Lawson-compact domain, then $X$ is a compact, core compact, core coherent web space. 
\end{corollary}
\begin{proof}
	By \namedref{Lemma}{lem:retract-six}, $\mathbb{X}_\bot^\top$ is a Scott-continuous retract of $L$. The function space $[X\to \mathbb{X}_\bot^\top]$ is a Scott-continuous retract of $[X\to L]$. As a Scott-continuous retract of a Lawson-compact domain is a Lawson-compact domain, we obtain that $[X\to \mathbb{X}_\bot^\top]$ is a Lawson-compact domain. Since $X$ is an RW$^*$-space, $[X\to\mathbf2]$ and $[X\to\Lambda]$ are continuous domains. The result follows by \namedref{Theorem}{thm:intro-main}.
\end{proof}

In \cite{LiangKeimel1999}, it is proved that a pointed Lawson-compact $L$-domain is an RB-domain. With this result at hand, we generalize the conclusion to general $sL$-domains, which need not have a least element.

\begin{proposition}\label{prop:swisrw}
	A Lawson-compact $sL$-domain is an RB-domain. As a consequence, a compact, core compact, core coherent web space is an RW$^*$-space.
\end{proposition}
\begin{proof}
		Let $L$ be a Lawson-compact sL-domain. Adjoin a least element $\bot$ and put
		$D=L_{\bot}=L\sqcup\{\bot\}$. Then $D$ is a
		pointed Lawson-compact L-domain, thus a pointed RB-domain (see \cite{LiangKeimel1999} and \cite{ZouLiGuo2018}). Hence there exists a directed family
		$(\delta_\alpha)_{\alpha\in A}$
		of Scott-continuous deflations on $D$
		such that $\bigvee_{\alpha\in A}\delta_\alpha
		=\id_D$. We shall show that a suitable cofinal subfamily of these deflations restricts to an approximate identity of
		deflations on $L$. 
		
		Since $L$ is Lawson-compact, there are finitely many minimal compact elements $m_1,\ldots,m_n$ such that
		$L=\bigcup_{i=1}^{n}\up_L m_i$. For each $i=1,\ldots,n$, we have
		$m_i=\bigvee_{\alpha\in A}
		\delta_\alpha(m_i)$. Compactness of $m_i$ in $D$ implies that there exists $\alpha_i\in A$
		such that $m_i\leq\delta_{\alpha_i}(m_i)$.
		Since $\delta_{\alpha_i}\leq\id_D$,
		it follows that
		$\delta_{\alpha_i}(m_i)=m_i$. By directedness, there exists $\alpha_0\in A$ with $\alpha_0\geq\alpha_i$ for all $1\leq i\leq n$. Therefore, for every $\alpha\geq\alpha_0$,
		$$m_i
		\leq \delta_{\alpha_0}(m_i)
		\leq\delta_\alpha(m_i)
		\leq m_i,$$
		and hence $\delta_\alpha(m_i)=m_i$ for all $1\leq i\leq n$. Now fix $\alpha\geq\alpha_0$ and
		$x\in L$. Since $L=\bigcup_{i=1}^{n}\up_L m_i$, there exists $i$ with $m_i\leq x$. By monotonicity,
		\[
		m_i=\delta_\alpha(m_i)
		\leq\delta_\alpha(x)
		\leq x.
		\]
		In particular, $\delta_\alpha(x)\neq\bot$.
		Thus $\delta_\alpha(L)\subseteq L$. For every $\alpha\geq\alpha_0$, define
		\[
		\varepsilon_\alpha
		=\delta_\alpha|_L:L\longrightarrow L.
		\]
		Because directed suprema in $L$ agree
		with those computed in $D$, each
		$\varepsilon_\alpha$ is Scott-continuous.
		Furthermore, $\varepsilon_\alpha(L)
		\subseteq\delta_\alpha(D)$, so $\varepsilon_\alpha$ has finite image. Hence every $\varepsilon_\alpha$ is a deflation on $L$. Put $A_0=\{\alpha\in A:\alpha\geq\alpha_0\}$. Then the family
		$(\varepsilon_\alpha)_{\alpha\in A_0}$
		is directed and, for every $x\in L$,
		\[
		\begin{aligned}
			\bigvee_{\alpha\in A_0}
			\varepsilon_\alpha(x)
			&=
			\bigvee_{\alpha\in A_0}
			\delta_\alpha(x)\\
			&=
			\bigvee_{\alpha\in A}
			\delta_\alpha(x)\\
			&=x.
		\end{aligned}
		\]
		Therefore $\bigvee_{\alpha\in A_0}
		\varepsilon_\alpha=\id_L$. Thus $L$ admits a directed approximate identity consisting of deflations. By definition, $L$ is an RB-domain.
		
		The second assertion is immediate by \namedref{Theorem}{thm:intro-main} and  \namedref{Theorem}{rw:thm:characterization}.
\end{proof}

	Now we use function space calculus to take a glance at Lawson-Mislove's problem. The following are some notations. As can be seen from \namedref{Table}{tab:mytable}, this problem is settled within $\mathcal{RW}^*\times \mathcal{RB}_\bot$.
	
	\quad
	
	\noindent\textbf{Notations:}
	\begin{itemize}
		\item $\mathcal{RW}^*$: the class of all RW$^*$-spaces;
		\item $\mathcal{SW}$: the class of all compact, core compact, core coherent web spaces;
		\item $\mathcal{RB}_{\bot}$: the class of all pointed RB-domains;
		\item $s\mathcal{RB}_{\bot}$: the class of all pointed Lawson-compact domains that are $sL$-domains;
	\end{itemize}
	
\begin{center}
	
	\begin{tikzpicture}[
		x=1cm,
		y=1cm,
		line width=0.8pt
		]
		
		\draw (0,0) rectangle (10,9);
		
		\draw (1,0) -- (1,6.8);
		\draw (2.3,0) -- (2.3,9);
		\draw (6.2,0) -- (6.2,8);
		
		\draw (0,6.8) -- (10,6.8);
		\draw (2.3,8) -- (10,8);
		\draw (1,4) -- (10,4);
		
		\draw (0,7.8) -- (2.3,6.8);
		\draw (1.2,9) -- (2.3,6.8);
		
		\node at (0.6,7.15) {Space};
		\node at (1,8) {$\checkmark$/$\times$};
		\node at (1.85,8.6) {Dcpo};
		\node at (0.5,3.2) {$\mathcal{RW}^*$};
		\node at (1.65,5.4)  {$\mathcal{SW}$};
		\node at (6.2,8.5) {$\mathcal{RB}_{\bot}$};
		\node at (4.25,7.4) {$s\mathcal{RB}_{\bot}$};
		\node at (4.25,5.4) {$\checkmark$ (\namedref{Theorem}{thm:intro-main})};
		\node at (8.1,5.4) {$\checkmark$ (\namedref{Theorem}{thm:intro-main})};
		\node at (4.25,2.1) {$\checkmark$ (\namedref{Theorem}{rw:thm:characterization})};
		\node at (8.1,2.1) {$\times$ (\namedref{Corollary}{cor:xmark})};
	\end{tikzpicture}
	
	\captionof{table}{Partial solution to Lawson-Mislove's problem}
	\label{tab:mytable}
	
\end{center}

\section{A compact, core compact, core coherent web space which is not a C-space}\label{sec:example}

A \emph{C-space}, or \emph{worldwide web space} in Ern\'e's terminology,
is a space in which every $x\in U$ open admits $y\in U$ with
$x\in\intr(\up y)\subseteq\up y\subseteq U$.
A \emph{wide web space} has, for every $x\in U$ open, a neighborhood
$V$ of $x$ such that every finite subset of $V$ has a lower bound in $U$.
These are characterized by complete distributivity and the wide-coframe
property of $\OO(X)$, respectively~\cite{Erne2009,Erne2019}.
In particular every C-space is a wide web space: the point $y$ in the
definition is a common lower bound for every subset of $\intr(\up y)$.

A complete lattice is \emph{supercompactly generated} if each element
is a join of elements $c$ for which $c\leq\bigvee S$ implies
$c\leq s$ for some $s\in S$, for every subset $S$. Such elements are called \emph{supercompact elements}. Recall that in \cite{LuanLiSG} we obtain the following theorem:

\begin{theorem}\label{that-thm}(\cite[Theorem 5.7]{LuanLiSG})
	Let $X$ be a topological space. The following assertions are equivalent.
	\begin{enumerate}
		\item The space $X$ is a compact, core coherent space with a supercompactly generated open set lattice. 
		\item The space $X$ is a compact, core coherent space with an algebraic open set lattice, and for each compact open set $U$ of $X$, the derived poset $\mathfrak{FD}(U)$ is a finitely generated upper set.  
		\item The space $X$ is a compact, core coherent space with an algebraic open set lattice, and each compact open set $U$ of $X$ has a finest decompostion.
		\item For each bifinite domain $P$, the function space $[X\rightarrow P]$ is a bifinite domain. 
		\item For each bifinite domain $P$, the function space $[X\rightarrow P]$ is a Lawson-compact algebraic domain. 
		\item The function spaces $[X\rightarrow \mathbf 2]$, $[X\rightarrow \Lambda]$ are algebraic domains, and the function space $[X\rightarrow \mathbb{X}^{\top}_{\bot}]$ is a Lawson-compact algebraic domain. 
	\end{enumerate}
\end{theorem}
It can be seen that \namedref{Theorem}{thm:intro-main} is an analogy of \namedref{Theorem}{that-thm} in the setting of continuous domains. From (i) of \namedref{Theorem }{that-thm}, we naturally propose the following conjecture:
\begin{quote}
	\itshape
	For a topological space $X$, $X$ is a compact, core compact, core coherent C-space if and only if for each RB-domain $L$, the function space $[X\rightarrow L]$ is an RB-domain. 
\end{quote}
Since a complete lattice is completely distributive if and only if it is both continuous and $\omega$-distributive (see \cite[Lemma 1 (5) ]{Erne2009}), comparing with (i) of \namedref{Theorem}{thm:intro-main}, this is exactly an open question proposed by
Ern\'e~\cite[Section~6, p.~2065]{Erne2009}:
\begin{quote}
	\itshape
	But it remains open whether all $4$-web spaces are already
	$\omega$-web spaces, and whether all $3$-distributive lattices are
	$\omega$-distributive ($=\mathcal F$-distributive).
\end{quote}
In this section, we answer this open problem in the negative. Hence the above conjecture fails.

\subsection{Distributivity and the semilattice construction}

Given a family $\mathcal A$ of subsets
of $L$, its \emph{crosscut system} is
\[
\mathcal A^{\#}
 =\left\{Z\subseteq\bigcup\mathcal A:
                 Z\cap A\ne\varnothing\text{ for every }A\in\mathcal A\right\}.
\]

\begin{lemma}\label{ex:lem:distribution}
Let $L$ be a complete lattice, $\kappa$ be $\infty$, $\omega$, or a cardinal number with $\kappa\ge 2$. The following are equivalent:
\begin{enumerate}
\item $L$ is $\kappa$-distributive;
\item for every $\mathcal A\subseteq\PP_\kappa L$,
\begin{equation}\label{ex:eq:crosscut}
\bigwedge_{A\in\mathcal A}\bigvee A
 =\bigvee_{Z\in\mathcal A^{\#}}\bigwedge Z;
\end{equation}
\item for every set $I$, sets $J_i$ with $|J_i|<\kappa$, and
elements $a_{i,j}\in L$, one has
\begin{equation}\label{ex:eq:matrix}
\bigwedge_{i\in I}\bigvee_{j\in J_i}a_{i,j}
 =\bigvee_{f\in\prod_{i\in I}J_i}\bigwedge_{i\in I}a_{i,f(i)}.
\end{equation}
For $\kappa=\omega$, each $J_i$ is a finite set; For $\kappa=\infty$, there is no restriction on the sets $J_i$.
\end{enumerate}
\end{lemma}

\begin{proof}
The right-hand side of \eqref{ex:eq:crosscut} is always at most its
left-hand side. Assume (i), and write the left-hand side as $b$.
For $a\in\DD_\kappa b$, the set
$Z_a=(\bigcup\mathcal A)\cap\up a$ meets every $A\in\mathcal A$,
by \eqref{eq:kappa-approx}. Thus $Z_a\in\mathcal A^\#$ and
$a\leq\bigwedge Z_a$. Taking the join over $\DD_\kappa b$ proves (ii).
This also covers an empty $\mathcal A$, with $Z_a=\varnothing$.

Conversely, assume (ii) and fix $b\in L$. Set
$\mathcal A_b=\{A\in\PP_\kappa L:b\leq\bigvee A\}.$
Since $\{b\}\in\mathcal A_b$, the left-hand side of
\eqref{ex:eq:crosscut} for this family is $b$.
For each $Z\in\mathcal A_b^\#$ and $A\in\mathcal A_b$, an element
$z\in Z\cap A$ satisfies $\bigwedge Z\leq z$. Hence
$\bigwedge Z\in\DD_\kappa b$. Formula~\eqref{ex:eq:crosscut} gives
$b\leq\bigvee\DD_\kappa b$; the reverse inequality follows by
using $A=\{b\}$ in \eqref{eq:kappa-approx}.

Finally, the equivalence of (ii) and (iii) follows from the Axiom of
Choice. For clarity, if all rows are nonempty, the set of selected
values $\{a_{i,f(i)}:i\in I\}$ is a crosscut of the family of row
sets. Conversely, for each crosscut $Z$, choose in each row a value
belonging to $Z$; the meet of those selected values is at least
$\bigwedge Z$. If a row is empty, both sides of
\eqref{ex:eq:matrix} are $0$. If $I=\varnothing$, both sides are $1$.
Taking the rows to be arbitrary members of $\mathcal A$ proves
the converse implication as well.
\end{proof}

The terminology in \cite[Sections~1 and~6]{Erne2009} therefore makes
$\omega$-distributivity exactly the wide coframe property, and
$\infty$-distributivity exactly complete distributivity in ZFC.
The coframe property is equivalent to \emph{fixed $3$-distributivity},
which requires \eqref{ex:eq:crosscut} for
$\mathcal A\subseteq\PP_3L$ with $\bigcap\mathcal A\ne\varnothing$.
We will establish the full condition of \eqref{eq:kappa-approx}.

\begin{lemma}\label{ex:lem:finite}
Let $Y$ be a $T_0$ space whose specialization order has binary meets.
If the meet map $Y\times Y\longrightarrow Y$ is continuous, then
$\OO(Y)$ is $\kappa$-distributive for every integer $\kappa\geq 2$.
\end{lemma}

\begin{proof}
Fix $\kappa\geq 2$ and set $n=\kappa-1$. The map
\[
\mu_n:Y^n\longrightarrow Y,
\qquad \mu_n(y_1,\ldots,y_n)=y_1\wedge\cdots\wedge y_n,
\]
is continuous, by induction, with $\mu_1$ the identity.
Let $U$ be a nonempty open set and $x\in U$. Continuity at
$(x,\ldots,x)$ gives open neighborhoods $V_1,\ldots,V_n$ of $x$
such that $\mu_n(\prod_{i=1}^n V_i)\subseteq U$.
For $V=\bigcap_{i=1}^n V_i$, one has
\begin{equation}\label{ex:eq:finite-local}
x\in V\subseteq U,\qquad \mu_n(V^n)\subseteq U.
\end{equation}

We claim that $V\in\DD_\kappa U$ in $\OO(Y)$.
Any family of fewer than $\kappa$ open sets covering $U$ is nonempty
and can be listed, with repetitions, as $U_1,\ldots,U_n$.
If $V\nsubseteq U_i$ for every $i$, choose $x_i\in V\setminus U_i$.
Then $z=\bigwedge_{i=1}^n x_i\in U$, so $z\in U_j$ for some $j$.
Since $z\leq x_j$ and $U_j$ is upper, this implies $x_j\in U_j$,
a contradiction. The claim follows from \eqref{eq:kappa-approx}.
Every $x\in U$ belongs to such a $V$, proving
$U=\bigvee\DD_\kappa U$. For $U=\varnothing$, the approximation
set is empty and the equality is immediate.
\end{proof}

We now specify the version of the Lawson construction used below;
the logarithmic formula is from \cite{GL2022}.
For real $t\geq 1$, define
$h(t)=\log_2\bigl(\log_2(1+t)\bigr).$
The function $h$ is increasing, $h(1)=0$, and
\begin{equation}\label{ex:eq:slow}
h(2k)-h(k)
 =\log_2\!\left(\frac{\log_2(1+2k)}{\log_2(1+k)}\right)
 \longrightarrow 0\qquad(k\longrightarrow\infty).
\end{equation}
Indeed, $1+2k<2(1+k)$ for $k\geq1$, and hence
\[
 0\leq h(2k)-h(k)
 <\log_2\!\left(1+\frac{1}{\log_2(1+k)}\right)
 \longrightarrow0.
\]
For $i\in\Nzero$, put
\begin{equation}\label{ex:eq:coordinates}
M_i=2^{2^i}-1,\qquad N_i=\{1,\ldots,M_i\},\qquad S_i=\PP(N_i).
\end{equation}
Order $S_i$ by inclusion, and define $\gamma_i:S_i\to[0,\infty]$ by
\begin{equation}\label{ex:eq:sigma}
\gamma_i(A)=
\begin{cases}
\infty,&A=N_i,\\
i-h(|N_i\setminus A|),&A\ne N_i.
\end{cases}
\end{equation}
Each $\gamma_i$ is monotone. Since $h(M_i)=i$, for proper $A$ we have
$0\leq\gamma_i(A)\leq i$, and in particular
\begin{equation}\label{ex:eq:sigma-special}
\gamma_i(\varnothing)=0,
\qquad \gamma_i(N_i\setminus\{k\})=i\quad(k\in N_i).
\end{equation}

\begin{lemma}\label{ex:lem:slow-intersection}
For real numbers $\tau>\varepsilon>0$, there is $i_0\in\Nzero$ such
that for all $i\geq i_0$ and $A,B\in S_i$,
\[
\gamma_i(A)>\tau,\quad\gamma_i(B)>\tau
\quad\Longrightarrow\quad
\gamma_i(A\cap B)>\tau-\varepsilon.
\]
\end{lemma}
\begin{proof}
By \eqref{ex:eq:slow}, choose an integer $K\geq1$ such that
$h(2k)<h(k)+\varepsilon$ for every integer $k\geq K$.
Choose $i_0>\tau-\varepsilon+h(2K)$. Fix $i\geq i_0$.
If $A=N_i$ or $B=N_i$, the conclusion follows immediately.
Otherwise, set
\[
a=|N_i\setminus A|,\quad b=|N_i\setminus B|,\quad
k=\max\{a,b\},\quad c=|N_i\setminus(A\cap B)|.
\]
Then $1\leq c\leq a+b\leq 2k$. If $k\geq K$, then
$$i-h(c)>i-h(k)-\varepsilon =\min\{\gamma_i(A),\gamma_i(B)\}-\varepsilon >\tau-\varepsilon.$$
If $k<K$, monotonicity of $h$ gives
$$i-h(c)\geq i-h(2K)\geq i_0-h(2K)>\tau-\varepsilon.$$
In both cases, $i-h(c)=\gamma_i(A\cap B)$.
\end{proof}

Let $\Rbar=[0,\infty]$ with its usual order and Scott-topology.
The nonempty proper open sets of $\Rbar$ are the intervals
$(r,\infty]$ with $0\leq r<\infty$. Give each finite lattice $S_i$
its upper-set topology, which equals its Scott-topology. Let
\begin{equation}\label{ex:eq:ambient}
K=\Rbar\times\prod_{i\in\Nzero}S_i
\end{equation}
with the product topology, and give
\begin{equation}\label{ex:eq:Z}
Z=\{(s,(A_i)_{i\in\Nzero})\in K:
                         s\leq\gamma_i(A_i)\text{ for all }i\in\Nzero\}
\end{equation}
the subspace topology. In this construction this is the specified
topology on $Z$; it is not replaced by a Scott-topology on $Z$.
A base consists of
\begin{equation}\label{ex:eq:base}
B(J;m,E)=Z\cap
 \left(J\times\prod_{i=0}^{m}\up E_i\times\prod_{i>m}S_i\right),
\end{equation}
where $m\in\Nzero$, $E_i\subseteq N_i$, and either $J=\Rbar$ or
$J=(r,\infty]$ for a real $r\geq0$.

\begin{proposition}\label{ex:prop:order}
The specialization order of $Z$ is the componentwise order:
\[
(s,(A_i))\leq(t,(B_i))
\quad\Longleftrightarrow\quad s\leq t\text{ and }A_i\subseteq B_i
\text{ for every }i.
\]
In particular, $Z$ is $T_0$. Its least and greatest elements are
\[
\bots=(0,(\varnothing)_i),\qquad
\tops=(\infty,(N_i)_i).
\]
Every nonempty open subset of $Z$ contains $\tops$.
\end{proposition}
\begin{proof}
The basic opens in \eqref{ex:eq:base} are upper sets for the
componentwise order, proving one implication. Conversely, if $s>t$,
then $t<\infty$ and the open cylinder with first coordinate in
$(t,\infty]$ contains $(s,(A_i))$ but not $(t,(B_i))$.
If $A_j\nsubseteq B_j$, choose $e\in A_j\setminus B_j$; the open
cylinder requiring $e$ in the $j$th coordinate separates the two
points in the same direction. This proves the order assertion and
the $T_0$ property. The formulas for the bounds follow from
\eqref{ex:eq:sigma-special} and $\gamma_i(N_i)=\infty$.
The last assertion follows because opens are upper sets.
\end{proof}

\begin{proposition}\label{ex:prop:meet}
For $x=(s,(A_i))$ and $y=(t,(B_i))$ in $Z$, their greatest lower
bound is
\begin{equation}\label{ex:eq:meet}
x\wedge y=(u,(A_i\cap B_i)),\qquad
u=\min\left\{s,t,\inf_{i\in\Nzero}\gamma_i(A_i\cap B_i)\right\}.
\end{equation}
\end{proposition}
\begin{proof}
The point on the right belongs to $Z$ and is below $x$ and $y$.
If $(v,(D_i))\in Z$ is another common lower bound, then
$v\leq s,t$ and $D_i\subseteq A_i\cap B_i$ for every $i$.
Monotonicity of $\gamma_i$ yields
$v\leq\gamma_i(D_i)\leq\gamma_i(A_i\cap B_i)$ for every $i$.
Thus $v\leq u$, proving the assertion.
\end{proof}

Give $\Rbar$ its compact order topology, denoted by $\Rbar^{\sharp}$,
and each $S_i$ its discrete topology, denoted by $S_i^{\sharp}$.
The map $s\mapsto s/(1+s)$, with $\infty\mapsto1$, identifies
$\Rbar^{\sharp}$ with $[0,1]$. The product
\[
K^{\sharp}=\Rbar^{\sharp}\times\prod_{i\in\Nzero}S_i^{\sharp}
\]
is compact Hausdorff by the Tychonoff theorem, and its topology
refines the topology of $K$.

\Needspace{8\baselineskip}
\begin{lemma}\label{ex:lem:closed}
The set $Z$ is closed in $K^{\sharp}$.
\end{lemma}
\begin{proof}
For every $i$, the subset
$$G_i=\bigcup_{A\in S_i}[0,\gamma_i(A)]\times\{A\} \subseteq\Rbar^{\sharp}\times S_i^{\sharp}$$
is closed, being a finite union of closed sets.
The constraint $s\leq\gamma_i(A_i)$ is the inverse image of $G_i$
under the corresponding coordinate projection. The intersection
of these closed constraints is $Z$.
\end{proof}

\begin{proposition}\label{ex:prop:compact}
The space $Z$ is compact and countably based. For each $x\in U\in\OO(Z)$,
there is a compact saturated set $C$ such that
\begin{equation}\label{ex:eq:compact-nbhd}
x\in\intr_Z C\subseteq C\subseteq U.
\end{equation}
In particular, $Z$ is locally compact.
\end{proposition}
\begin{proof}
\namedref{Lemma}{ex:lem:closed} and the coarseness of the topology of $Z$
give compactness. In \eqref{ex:eq:base} one may restrict $r$ to
nonnegative rational numbers. There are countably many resulting
basic opens, since the coordinate spaces are finite.

Choose $B(J;m,E)$ with $x\in B(J;m,E)\subseteq U$.
If $J=\Rbar$, put $C=B(J;m,E)$. This set is closed in $K^{\sharp}$:
the conditions $A_i\supseteq E_i$ constrain finitely many finite
discrete coordinates. It is therefore compact in $Z$, and it is open.
If $J=(r,\infty]$ and the first coordinate of $x$ is $s$, choose a
finite real number $t$ with $r<t<s$. Such a choice also exists when
$s=\infty$. Set
\[
C=Z\cap\left([t,\infty]\times\prod_{i=0}^{m}\up E_i
                               \times\prod_{i>m}S_i\right).
\]
This set is closed in $K^{\sharp}$, hence compact in $Z$, and
$$x\in B((t,\infty];m,E)\subseteq C\subseteq B(J;m,E)\subseteq U.$$
In both cases $C$ is upper for the specialization order, so it is
saturated. 
\end{proof}

It is well-known that any locally compact space is core compact. So we obtain:

\begin{corollary}\label{ex:prop:frame}
	The complete lattice $\OO(Z)$ is a continuous frame.
\end{corollary}

\begin{proposition}\label{prop:example-coherent}
The space $Z$ is coherent and core coherent.
\end{proposition}
\begin{proof}
Write $Z^\sharp$ for $Z$ with the compact Hausdorff topology induced
by $K^\sharp$. If $c\nleq z$ in $Z$, there is an open neighborhood
$O$ of $c$ whose closure in $Z^\sharp$ omits $z$.
Indeed, either a finite coordinate contains an element absent at $z$,
in which case a clopen coordinate cylinder works, or the first
coordinate of $c$ is larger than that of $z$, in which case choose
a strict intermediate threshold.

Let $C$ be compact saturated in $Z$ and $z\notin C$.
For every $c\in C$, saturation gives $c\nleq z$, so choose such
an $O_c$. Finitely many $O_c$ cover $C$. The union of their
$Z^\sharp$-closures is closed and omits $z$. Its complement
is a $Z^\sharp$-open neighborhood of $z$ disjoint from $C$.
As this holds for every $z\notin C$, the set $C$ is closed in
$Z^\sharp$. Intersections of two compact saturated sets are thus
$Z^\sharp$-closed and hence compact in $Z$, proving coherence.

For $U\ll V$ in $\OO(Z)$, local compactness provides a directed
cover of $V$ by finite unions of interiors of compact saturated
neighborhoods contained in $V$. Some such union contains $U$.
The corresponding finite union of compact neighborhoods is a compact
saturated $C$ with $U\subseteq C\subseteq V$.
Given $U_i\ll V_i$ for $i=1,2$, choose $C_i$ this way.
The compact set $C_1\cap C_2$ lies between $U_1\cap U_2$ and
$V_1\cap V_2$. Every directed open cover of the latter covers the
compact set with one member, proving
$U_1\cap U_2\ll V_1\cap V_2$.
\end{proof}

To apply the  $\kappa$-distributivity criterion, it remains to verify
continuity of the semilattice operation on $Z$.

\begin{proposition}\label{ex:prop:continuous-meet}
The map $\wedge:Z\times Z\longrightarrow Z$ in \eqref{ex:eq:meet}
is continuous.
\end{proposition}
\begin{proof}
Let $x=(s,(A_i))$, $y=(t,(B_i))$, and suppose
$x\wedge y\in W=B(J;m,E)$. Write the first coordinate of
$x\wedge y$ as $u$. For $i\leq m$, we have $E_i\subseteq A_i\cap B_i$.
If $J=\Rbar$, then $x,y\in W$ and every meet of two points of $W$
belongs to $W$. Hence $W\times W\subseteq\wedge^{-1}(W)$.

Suppose $J=(r,\infty]$. Since $u>r$, choose finite real numbers
$\tau,\varepsilon$ such that
\begin{equation}\label{ex:eq:margin}
r<\tau-\varepsilon<\tau<u,\qquad\varepsilon>0.
\end{equation}
Then $\tau>\varepsilon$, $s,t>\tau$, and
$\gamma_i(A_i\cap B_i)>\tau$ for every $i$.
Choose $i_0$ from \namedref{Lemma}{ex:lem:slow-intersection}, and put
$q=\max\{m,i_0\}$. The sets
\begin{align*}
P&=Z\cap\left((\tau,\infty]\times\prod_{i=0}^{q}\up A_i\times\prod_{i>q}S_i\right),\\
Q&=Z\cap\left((\tau,\infty]\times\prod_{i=0}^{q}\up B_i\times\prod_{i>q}S_i\right)
\end{align*}
are open neighborhoods of $x$ and $y$.
Take $x'=(s',(A_i'))\in P$ and $y'=(t',(B_i'))\in Q$.
For $i\leq q$, monotonicity gives
$\gamma_i(A_i'\cap B_i')\geq\gamma_i(A_i\cap B_i)>\tau.$
For $i>q$, the defining constraints of $Z$ imply
$\gamma_i(A_i')\geq s'>\tau$ and
$\gamma_i(B_i')\geq t'>\tau$. \namedref{Lemma}{ex:lem:slow-intersection}
therefore gives $\gamma_i(A_i'\cap B_i')>\tau-\varepsilon$.
Consequently
$$\inf_i\gamma_i(A_i'\cap B_i')\geq\tau-\varepsilon,$$
so the first coordinate of $x'\wedge y'$ is at least
$\tau-\varepsilon>r$. Also $E_i\subseteq A_i'\cap B_i'$ for
$i\leq m$. Hence $x'\wedge y'\in W$, proving
$P\times Q\subseteq\wedge^{-1}(W)$.
\end{proof}

\begin{corollary}\label{ex:cor:finite-dist}
The continuous frame $L=\OO(Z)$ is $\kappa$-distributive for every
integer $\kappa\geq2$.
\end{corollary}
\begin{proof}
Apply \namedref{Lemma}{ex:lem:finite} to \namedref{Propositions}{ex:prop:order},
\ref{ex:prop:meet}, and~\ref{ex:prop:continuous-meet}, and use
\namedref{Corollary}{ex:prop:frame}.
\end{proof}

\subsection{A solution to Ern\'e's problem}

The preceding corollary does not settle $\omega$-distributivity.
We now construct finite covers of unbounded cardinality to exhibit
the obstruction.

Let
\begin{equation}\label{ex:eq:Ustar}
U_* =\{(s,(A_i))\in Z:s>1\}.
\end{equation}
This is a nonempty open set, since $\tops\in U_*$.
For $n\geq2$ and $k\in N_n$, define
\begin{equation}\label{ex:eq:z}
z_{n,k}=(n,(A_i^{n,k})),\qquad
A_i^{n,k}=\begin{cases}
N_n\setminus\{k\},&i=n,\\
N_i,&i\ne n.
\end{cases}
\end{equation}
Formula~\eqref{ex:eq:sigma-special} gives $z_{n,k}\in Z$.
The finite first coordinate $n$ is essential for this membership.

\begin{lemma}\label{ex:lem:escape}
For every neighborhood $V$ of $\tops$, there is $n_0\geq2$ such that
$z_{n,k}\in V$ for all $n\geq n_0$ and all $k\in N_n$.
For each $n\geq2$, the finite family $\{z_{n,k}:k\in N_n\}$ has
no common lower bound in $U_*$.
\end{lemma}
\begin{proof}
Choose a basic open $B(J;m,E)$ containing $\tops$ and contained
in $V$. If $J=\Rbar$, every $n>m$ with $n\geq2$ suffices. If
$J=(r,\infty]$, every $n>\max\{m,r\}$ with $n\geq2$ suffices.
Indeed, all coordinates constrained by this basic open equal
$N_i$ at $z_{n,k}$, and the first coordinate is $n$.

If $y=(s,(D_i))$ is below all $z_{n,k}$, then
$$D_n\subseteq\bigcap_{k\in N_n}(N_n\setminus\{k\})=\varnothing.$$
The defining inequality for $Z$ gives
$s\leq\gamma_n(\varnothing)=0$. Thus $y\notin U_*$.
\end{proof}

For each $n\geq2$ and $k\in N_n$, the open coordinate cylinder
$$H_{n,k}=\{(s,(A_i))\in Z:k\in A_n\}$$
also has the description
\begin{equation}\label{ex:eq:complement}
H_{n,k}=Z\setminus\dn z_{n,k}.
\end{equation}
Indeed, if $k\notin A_n$, then $A_n\subseteq N_n\setminus\{k\}$
and $s\leq\gamma_n(A_n)\leq n$, so $(s,(A_i))\leq z_{n,k}$.
The converse follows from the $n$th coordinate comparison.
Define the following entries of $L=\OO(Z)$:
\begin{equation}\label{ex:eq:entries}
a_{n,k}=U_*\cap H_{n,k}\qquad(n\geq2,\ k\in N_n).
\end{equation}

\begin{theorem}\label{ex:thm:countable}
The matrix \eqref{ex:eq:entries} satisfies
\begin{align}
\bigwedge_{n\geq2}\bigvee_{k\in N_n}a_{n,k}&=U_*\ne\varnothing,
\label{ex:eq:left}\\
\bigvee_{f\in\prod_{n\geq2}N_n}\bigwedge_{n\geq2}a_{n,f(n)}
 &=\varnothing.                                    \label{ex:eq:right}
\end{align}
In particular, $L$ is not $\omega$-distributive.
\end{theorem}
\begin{proof}
If $(s,(A_i))\in U_*$, then $A_n\ne\varnothing$ for every $n$,
since otherwise $s\leq\gamma_n(\varnothing)=0$.
For fixed $n\geq2$, choose $k\in A_n$. This proves
$\bigcup_{k\in N_n}a_{n,k}=U_*$, and hence \eqref{ex:eq:left}.

Fix $f\in\prod_{n\geq2}N_n$, and put
\[
b_f=\bigwedge_{n\geq2}a_{n,f(n)}
    =\intr_Z\left(\bigcap_{n\geq2}a_{n,f(n)}\right).
\]
If $b_f$ were nonempty, \namedref{Proposition}{ex:prop:order} would give
$\tops\in b_f$. \namedref{Lemma}{ex:lem:escape}, applied to the open
neighborhood $b_f$, would then give $z_{n,f(n)}\in b_f$ for all
sufficiently large $n$. But $f(n)\notin A_n^{n,f(n)}$, so
$z_{n,f(n)}\notin a_{n,f(n)}$, contradicting
$b_f\subseteq a_{n,f(n)}$. Thus $b_f=\varnothing$ for every $f$,
which proves \eqref{ex:eq:right}. Each $N_n$ is finite and nonempty,
so \namedref{Lemma}{ex:lem:distribution} proves the last assertion.
\end{proof}

Since a complete lattice is completely distributive if and only if it is both continuous and $\omega$-distributive ($=$ wide coframe $=$ $\cF$-distributive), we have the following result. 

\begin{corollary}\label{ex:cor:approxzero}
$L$ is not a wide coframe and is not completely
distributive.
\end{corollary}

	The corresponding topological conclusion can be proved directly in
Ern\'e's terminology~\cite[Section~6]{Erne2009}. For a subset $A$ of a poset,
write $A^{\downarrow}=\{y:y\leq a\text{ for all }a\in A\}$.

Let $\kappa$ be $\omega$ or a cardinal number with $\kappa\ge 2$. Recall that a space $Y$ is a \emph{$\kappa$-web space} if, for every $x\in Y$
and every neighborhood $U$ of $x$, there is a neighborhood $V$ of
$x$ such that
\[
x\in A\subseteq V,\quad |A|<\kappa
\quad\Longrightarrow\quad A^{\downarrow}\cap U\ne\varnothing.
\]
For $\kappa=\omega$, $A$ is required to be finite. An $\omega$-web space is exactly a \emph{wide web space}.

\begin{proposition}\label{ex:prop:web}
	The space $Z$ is a $\kappa$-web space for every integer
	$\kappa\geq2$, but is not a wide web space.
\end{proposition}
\begin{proof}
	 The space $Z$ is not a wide web space since $\OO(Z)$ is not a wide coframe (=$\omega$-distributive) by  \namedref{Theorem}{ex:thm:countable}. By \namedref{Corollary}{ex:cor:finite-dist}, the continuous frame $L=\OO(Z)$ is $\kappa$-distributive for every
	 integer $\kappa\geq2$, hence $\kappa$-web by \cite[Theorem 6]{Erne2009}. 
\end{proof}

Combining the above results, we have the following theorem.

\begin{theorem}\label{thm:counterexample}
The space $Z$ is countably based, compact, core compact, core coherent, and a web space, but it is not a wide web space or a C-space.
Its open-set lattice is a continuous frame which is
$\kappa$-distributive for every finite $\kappa\geq2$ and is not
$\omega$-distributive. Nevertheless, $[Z\to L]$ is an RB-domain
for every RB-domain $L$.
\end{theorem}
\begin{proof}
Compactness and core compactness are \namedref{Proposition}{ex:prop:compact};
core coherence is \namedref{Proposition}{prop:example-coherent}.
\namedref{Corollary}{ex:prop:frame} gives continuity of $\OO(Z)$, and
\namedref{Corollary}{ex:cor:finite-dist} gives finite-cardinal distributivity.
\namedref{Theorem}{thm:web-equivalences} then makes $Z$ a web space.
The failure of $\omega$-distributivity is
\namedref{Theorem}{ex:thm:countable}, while \namedref{Proposition}{ex:prop:web}
shows that $Z$ is not a wide web space. Hence it is not
a C-space. The last assertion is \namedref{Theorem}{thm:intro-main}.
\end{proof}

\section{Lawson-compact domains are retracts of Lawson-compact algebraic domains}

The question, raised by Hui Kou, asks whether every Lawson-compact domain is a Scott-continuous retract of a Lawson-compact algebraic domain. Previously, by the result of Jia et al. \cite{JiaLiLuan25}, it was only known that the statement holds for countable Lawson-compact domains. Here we give an affirmative answer together with an explicit construction.

For a poset $P$ and a subset $S\subseteq P$, put
\[
S^{\up}=\{p\in P:(\forall s\in S)\ s\leq p\},
\qquad \mub_P(S)=\Min (S^{\up}).
\]
Recall that $P$ has \emph{property M} \cite{Jung1988} if, for every finite $F$, including $F=\varnothing$,
\begin{equation}\label{cov:eq:M}
	\mub_P(F)\text{ is finite},\qquad
	F^{\up}=\up_P\mub_P(F).
\end{equation}
In particular, $\Min P$ is finite and $P=\up_P\Min P$. 

Let $P$ be a poset with property M. A subset $C\subseteq P$ is \emph{mub-closed} if
\begin{equation}\label{eq:closure}
	\Min P\subseteq C,\qquad
	\dn_P C=C,\qquad
	p,q\in C\Longrightarrow\mub_P(p,q)\subseteq C.
\end{equation}
The set of all mub-closed subsets of $P$ is closed under arbitrary intersections. Thus, for each $S\subseteq P$, there is a smallest mub-closed subset $\Cl(S)$ containing $S$, i.e., the intersection of all mub-closed subset containing $S$. The closure operator ``$\Cl$'' is then monotone and idempotent. Fixed points of the operator ``$\Cl$'' are exactly all mub-closed sets.

Fix a Lawson-compact domain $D$ throughout this section. Interiors are Scott-interiors.

\subsection{Labelled posets over Lawson-compact domains}

\begin{lemma}\label{lem:interior}
	For every $a\in D$, $\intrs(\up a)=\Up a$. For every family of upper
	sets $(K_i)_{i\in I}$ in $D$,
	\begin{equation}\label{eq:interior-union}
		\intrs\Bigl(\bigcup_{i\in I}K_i\Bigr)
		=\bigcup_{i\in I}\intrs K_i.
	\end{equation}
\end{lemma}
\begin{proof}
	The Scott-open set $\Up a$ is contained in $\up a$. If $x\in\intrs(\up a)$,
	Scott-openness applied to $x=\sup\Down x$ gives $b\ll x$ with $a\le b$,
	so $a\ll x$. This proves the first equality.
	For the second, let $x$ belong to the left side. Choose $b\ll x$ in
	that Scott-open set and then $i$ with $b\in K_i$. Upperness gives
	$x\in\Up b\subseteq K_i$, whence $x\in\intrs K_i$.
	The reverse inclusion is immediate.
\end{proof}

The main tool for constructing the desired Lawson-compact algebraic domain is the labelled poset.

\begin{definition}\label{def:label}
	A \emph{labelled poset over $D$} is a poset $P$ with assigned data
	$(a_p,K_p)$ for each $p\in P$, where $a_p\in D$ and $K_p$ is a
	Lawson-closed upper set, such that
	\begin{equation}\label{eq:label}
		K_p\subseteq\Up a_p,\qquad
		p\le q\Longrightarrow a_p\le a_q\ \text{and}\ K_q\subseteq K_p.
	\end{equation}
	Write $U_p=\intrs K_p$. 
\end{definition}

For the convenience of later arguments, the property M on finite subsets can be reduced to the cases of the empty set and two-element sets.

\begin{lemma}\label{lem:M}
	For a poset $P$, the following assertions are equivalent.
	\begin{enumerate}
		\item $P$ has property M;
		\item $P$ is finitely generated as an upper set, and for any $x,y\in P$, $\mub_P\{x,y\}$ is finite and each $z\in \{x,y\}^{\up}$ is above some $w\in \mub_P\{x,y\}$. 
	\end{enumerate} 
\end{lemma}
\begin{proof}
	Only the ``$\Leftarrow$-part'' is nontrivial. Assume (ii). Then $\varnothing^{\up}=P$ is a finitely generated upper set. Let $\{p_1,p_2,\ldots,p_s\}\subseteq P$. Put $E_0=\Min P$. Suppose $E_n$ has been defined, set 
	$$E_{n+1}:=\bigcup_{e\in E_n}\mub \{e,p_{n+1}\}.$$
	Then $E_s$ is a finite subset of $\{p_1,p_2,\ldots,p_s\}^{\up}$ such that 
	$\{p_1,p_2,\ldots,p_s\}^{\up}={\up E_s}$. The finite subset $\Min E_s$ exactly contains all minimal upper bounds of $\{p_1,p_2,\ldots,p_s\}$ and $$\{p_1,p_2,\ldots,p_s\}^{\up}={\up \Min E_s}.$$
	Hence, $P$ has property M.
\end{proof}

The following lemma is the well-known  ``$\frac{2}{3}$-SFP theorem''.  

\begin{lemma}\label{2/3-sfp}
		An algebraic domain $L$ is Lawson-compact if and only if the poset $K(L)$ of compact elements has property M.
\end{lemma}

A subset $I$ of a poset $P$ is an \emph{ideal} if $I$ is downward closed and directed. Denote by $\Idl(P)$ the poset of all ideals of $P$ ordered by inclusion. Then $\Idl(P)$ is an algebraic domain. Compact elements of $\Idl(P)$ are exactly all principal ideals. The directed suprema is the set union (see \cite{CLD}).

\begin{proposition}\label{prop:reduction}
	Suppose there is a labelled poset $P$ over $D$ such that
	 \begin{enumerate}
	 	\item $P$ has property $M$;
	 	\item $\bigcup_{r\in \Min P}K_r=D$;
	 	\item $K_p\cap K_q=\bigcup_{r\in \mub\{p,q\}}K_r$ for all $p,q\in P$;
	 	\item $y\ll x\ \text{in}\ D\Longrightarrow
	 	(\exists p\in P)\ y\le a_p\ \text{and}\ x\in U_p.$
	 \end{enumerate}
Then the maps
	\begin{equation}\label{eq:retraction}
		s(x)=\{p:x\in U_p\},\qquad r(I)=\sup_{p\in I}a_p
	\end{equation}
	are Scott-continuous maps between $D$ and $\Idl(P)$ with $rs=\id_D$. Hence $D$ is a Scott-continuous retract of some Lawson-compact algebraic domain. 
\end{proposition}
\begin{proof}
	Let $x\in D$. Take $y\ll x$. There is some $r\in \Min P$ with $y\in K_r$. Then $x\in \intrs K_r=U_r$. Hence $r\in s(x)$. Suppose $p,q\in s(x)$. Then by (iii) and \namedref{Lemma}{lem:interior}, 
	$$x\in \intrs K_p\cap \intrs K_q=\intrs (K_p\cap K_q)=\intrs \bigcup_{r\in \mub \{p,q\}}K_r=\bigcup_{r\in \mub \{p,q\}}\intrs K_r.$$
	So there is some $r\in \mub\{p,q\}$ with $r\in s(x)$. It shows that
	$s(x)$ is directed; \eqref{eq:label} makes it lower. It is obvious that the map $s$ is monotone. For every directed $E\subseteq D$, suppose $p\in s(\sup E)$, that is, $\sup E\in U_{p}$. There is some $e_0\in E$ with $e_0\in U_p$. Hence $p\in s(e_0)\subseteq \bigcup_{e\in E}s(e)$. This implies that $s$ is Scott-continuous. 
	
	Since the value map is monotone, $\{a_p:p\in I\}$ is directed, so $r$ is well-defined. It is clear that $r$ is monotone. For a directed family of ideals $(I_j)$, each $p\in\bigcup_jI_j$ lies
	in some $I_j$, and hence $r\Bigl(\bigcup_jI_j\Bigr)=\sup_jr(I_j)$. Thus $r$ is Scott-continuous.
	
	If $p\in s(x)$ then $a_p\le x$, so $r(s(x))\le x$. For each $y\ll x$, by (iv), take a $p\in P$ with $y\le a_p$ and $x\in U_p$. Then $p\in s(x)$. So $r(s(x))\ge a_p\ge y$. The continuity of $D$ gives $r(s(x))\ge x$. Hence $rs=\id_D$. 
\end{proof}

\begin{lemma}\label{lem:buffer}
	Let $K$ be a Lawson-closed upper subset of $D$, and let $V\subseteq D$
	be finite, with $K\subseteq\bigcap_{a\in V}\Up a$. There are finitely
	many pairs $(c_j,d_j)$ such that
	\begin{equation}\label{eq:buffer}
		V\subseteq\dn c_j,\quad c_j\ll d_j,\quad
		K=\bigcup_jK_j,\quad K_j=K\cap\up d_j\subseteq\Up c_j.
	\end{equation}
	Each $K_j$ is a Lawson-closed upper set. Moreover
	$\intrs K=\bigcup_j\intrs K_j$.
\end{lemma}
\begin{proof}
	If $K=\varnothing$, use the empty family. Otherwise, for each $x\in K$,
	finite interpolation supplies $c_x,d_x$ with
	$V\subseteq\dn c_x$ and $c_x\ll d_x\ll x$.
	The sets $\Up d_x$ form a Lawson-open cover of the compact set $K$.
	Choose a finite subcover. Its pairs satisfy \eqref{eq:buffer} since
	$\Up d_j\subseteq\up d_j$ and $c_j\ll d_j$.
	The sets $\up d_j$ are Lawson-closed, and the interior identity follows
	from \eqref{eq:interior-union}.
\end{proof}

\subsection{Presheaves, element posets and their extensions}

In this subsection, we construct the desired Lawson-compact algebraic domain and show that $D$ is its retract.

Denote by $\FinSet$ the category of all finite sets and mappings. Regard a poset $\mathcal{J}$ as a category with an arrow $I\to J$ when $I\le J$.
A finite-set-valued presheaf $E:\mathcal J^{op}\to\FinSet$ is a functor from the opposite of $\mathcal{J}$ to $\FinSet$ \cite{MacMoe92}. It consists
of a finite set $E_J$ for each $J\in \mathcal{J}$ and maps $\rho_I^J:E_J\to E_I$ for each $I\le J$ in $\mathcal{J}$,
with
\begin{equation}\label{eq:functor}
	\rho_J^J=\id_{E_J},\qquad
	\rho_H^I\rho_I^J=\rho_H^J\quad(H\le I\le J).
\end{equation}
Its \emph{element poset} is the disjoint union $\int E=\coprod_JE_J$,
ordered by
\begin{equation}\label{eq:element-order}
	e\le f\iff I\le J\ \text{and}\ \rho_I^J(f)=e
	\qquad(e\in E_I,\ f\in E_J).
\end{equation}
The identities \eqref{eq:functor} give reflexivity and transitivity; antisymmetry follows from antisymmetry of the index order.

\begin{lemma}\label{lem:elements}
	Suppose $\mathcal J$ is a join-semilattice with a least element $\bot$ and $E:\mathcal{J}^{op}\to \FinSet$ is a presheaf. Then the poset $\int E$ has property M.
	For $e\in E_I$ and $f\in E_J$,
	\begin{equation}\label{eq:presheaf-mub}
		\mub_{\int E}\{e,f\}
		=\{g\in E_{I\vee J}:\rho_I^{I\vee J}(g)=e,\ 
		\rho_J^{I\vee J}(g)=f\}.
	\end{equation}
	In addition,
	\begin{equation}\label{eq:presheaf-min}
	  \Min \int E=\{e:e\in E_\bot\}.
	\end{equation}
	Furthermore, if $\dn J$ is finite for some $J\in \mathcal{J}$. Then $Q_J=\coprod_{I\le J}E_I$ is a finite mub-closed set.
\end{lemma}
\begin{proof}
 Every element on the right
	of \eqref{eq:presheaf-mub} is a minimal upper bound: an upper bound below it must
	have index both above and below $I\vee J$, and hence is the same label. Assume $h\in E_K$ is an upper bound of $e,f$. Then $K\ge I,J$. Hence $\rho^{K}_{I\vee J}(h)\in E_{I\vee J}$ is below $h$ and 
	$$\rho^{I\vee J}_I(\rho^{K}_{I\vee J}(h))=\rho^{K}_{I}(h)=e,\quad\quad \rho^{I\vee J}_J(\rho^{K}_{I\vee J}(h))=\rho^{K}_{J}(h)=f.$$
	It implies that each upper bound of $e$ and $f$ is above some element on the right
	of \eqref{eq:presheaf-mub}. So equation \eqref{eq:presheaf-mub} holds. The equation \eqref{eq:presheaf-min} can be proved similarly. By \namedref{Lemma}{lem:M}, $\int E$ has property M. The rest is straightforward. 
\end{proof}

\begin{definition}
	Given a presheaf $E:\mathcal J^{op}\to\FinSet$, a \emph{compatible tuple of $E$} is a family $\mathbf e=(e_J)_{J\in \mathcal J}$ satisfying the following conditions.
	\begin{enumerate}
		\item For all $J\in \mathcal{J}$, $e_J\in E_J$;
		\item For all $I\le J$ in $\mathcal{J}$, $\rho^J_I(e_J)=e_I$.
	\end{enumerate} 
\end{definition}

Let $P$ be a labelled poset over $D$ satisfying the following conditions:
\begin{enumerate}
	\item $P$ has property M;
	\item $\Cl(F)$ is finite for each finite subset $F\subseteq P$;
	\item $\bigcup_{r\in \Min P}K_r=D$;
	\item $T_0\subseteq P$ is a distinguished lower subset whose pairs
	already satisfy
	\begin{equation}\label{eq:settled}
		K_p\cap K_q=\bigcup_{c\in\mub_P(p,q)}K_c
		\qquad(p,q\in T_0).
	\end{equation}
\end{enumerate}

Let $\cF$ be the set of finite mub-closed subsets of $P$. It is a lattice with a least element $F_0$:
\begin{equation}\label{eq:index-lattice}
	F\wedge G=F\cap G,\qquad
	F\vee G=\Cl(F\cup G),\qquad F_0=\Cl(\varnothing).
\end{equation}
Every principal ideal of $\cF$ is finite. For $F\in\cF$ and $x\in D$, put
\begin{equation}\label{eq:traces}
	S_F(x)=\{p\in F:x\in K_p\},\qquad
	\cT(F)=\{S_F(x):x\in D\}.
\end{equation}
These are finite sets. 

\begin{lemma}\label{property-of-S}
	$\cT:\cF^{op}\to\FinSet$ defines a presheaf. Every $S\in\cT(F)$ is
	lower in $P$, meets $\Min P$, and satisfies
	\begin{equation}\label{eq:trace-compat}
		p,q\in S\cap T_0\Longrightarrow S\cap\mub_P(p,q)\ne\varnothing.
	\end{equation}
\end{lemma}
\begin{proof}
	Suppose $G\subseteq F$ in $\cF$. For $S_F(x)=S\in \cT(F)$, define $\cT(G\subseteq F)(S):=S\cap G$. One verifies that $S\cap G=S_G(x)$. So $\cT(G\subseteq F)$ defines a map from $\cT(F)$ to $\cT(G)$. It is then clear that $\cT$ defines a functor. 
	
	Lowerness follows from \eqref{eq:label}. It meets $\Min P$ because of the root cover, and
	\eqref{eq:trace-compat} from \eqref{eq:settled} and closedness of $F$.
\end{proof}

Given a poset $L$ and an element $x\in L$, let 
$x^<:={\downarrow x}\setminus \{x\}$.

For each $n\in \Nzero$, let
$$\cF_n:=\{F\in \cF: |F|\le |F_0|+n\}$$
be the subposet of $\cF$ with the induced order. Then $\cF_0$ is the poset $\{F_0\}$ with one element. In addition, 
$$\cF_0\subseteq \cF_1\subseteq \cF_2\subseteq\cdots$$
and $\cF=\bigcup_{n\in \Nzero}\cF_n$. We define finite presheaves $E:\cF^{op}\to \FinSet$ inductively for each $n\in \Nzero$ as follows.

\begin{itemize}
	\item For each specified $S\in \cT(F_0)$, set
	\begin{align}
		\widehat K(F_0,S,\varnothing)
		&=\bigcap_{p\in S}K_p,\label{eq:Khat}\\
		V(F_0,S,\varnothing)
		&=\{a_p:p\in S\}.\label{eq:Vhat}
	\end{align}
	The first set is a Lawson-closed upper
	set contained in $\bigcap_{a\in V(F_0,S,\varnothing)}\Up a$.
	Applying \namedref{Lemma}{lem:buffer}, there is a finite set $J=E_{F_0}[S,\varnothing]$ and  $(a_j,K_j)_{j\in J}$, where $a_j\in D$, $K_j$ Lawson-closed upper, $K_j\subseteq {\Up a_j}$, and
	\begin{equation}\label{eq:new-label}
		\begin{gathered}
			V(F_0,S,\varnothing)\subseteq {\downarrow a_j},\qquad \widehat K(F_0,S,\varnothing)=\bigcup_j K_j.
		\end{gathered}
	\end{equation}
	Let $E_{F_0}$ be the disjoint union 
	\begin{align}
	&	\bigsqcup \{E_{F_0}[S,\varnothing]~|~S\in \cT(F_0)\}.
	\end{align}
	Let $\rho^{F_0}_{F_0}:E_{F_0}\to E_{F_0}$ be the identity map. Let $\pi_{F_0}(e)$ be the unique $S\in \cT(F_0)$ with $e\in E_{F_0}[S,\varnothing]$ for each $e\in E_{F_0}$. 
	\item  Given $F\in \cF_{n+1}\setminus \cF_n$, let $E|_{F^<}$ be the presheaf restricted to the full subcategory $F^<$. For each specified $S\in \cT(F)$ and each compatible tuple $\mathbf e=(e_I)_{I<F}$ of the presheaf $E|_{F^<}$ with
	\begin{align}
		& \pi_I(e_I)=S\cap I\ \ (I<F),
	\end{align}
	set
	\begin{align}
		\widehat K(F,S,\mathbf e)
		&=\bigcap_{p\in S}K_p\cap
		\bigcap_{I< F}K_{e_I},\label{eq:Khat1}\\
		V(F,S,\mathbf e)
		&=\{a_p:p\in S\}\cup\{a_{e_I}:I<F\}.\label{eq:Vhat1}
	\end{align}
	The first set is a Lawson-closed upper
	set contained in $\bigcap_{a\in V(F,S,\mathbf e)}\Up a$.
	Applying \namedref{Lemma}{lem:buffer}, there is a finite set $J=E_F[S,\mathbf e]$ and $(a_j,K_j)_{j\in J}$, where $a_j\in D$, $K_j$ Lawson-closed upper, $K_j\subseteq {\Up a_j}$, and
	\begin{equation}\label{eq:new-label}
		\begin{gathered}
			V(F,S,\mathbf e)\subseteq {\downarrow a_j},\qquad \widehat K(F,S,\mathbf e)=\bigcup_j K_j.
		\end{gathered}
	\end{equation}
	Let $E_F$ be the disjoint union of these $E_F[S,\mathbf e]$ over all choices of such $S$ and $\mathbf e$. Then $E_F$ is finite. Put
	\begin{equation}
		\begin{gathered}
						\rho_I^F(e)=e_I\ (I<F,e\in E_F(S,\mathbf e)),\qquad\rho_F^F(e)=e.
		\end{gathered}
	\end{equation}
	Let $\pi_F(e)$ be the unique $S\in \cT(F)$ with $e\in E_F[S,\mathbf e]$.  
\end{itemize}

 
Notice that in the construction above, for each $e\in \int E$, a pair $(a_e,K_e)$ with $a_e\in D$, $K_e\subseteq D$ Lawson-closed upper, $K_e\subseteq {\Up a_e}$ is assigned. 

\begin{remark}
	The first step of the above inductive construction, namely the step concerning $E_{F_0}$, can be omitted; one only needs to treat the class of compatible tuples as the empty set within the second inductive step. We present it separately chiefly for clarity.
\end{remark} 

From the construction above, one easily sees that:

\begin{lemma}\label{natural}
	The correspondence $\pi$ is a natural transformation from $E$ to $\cT$. 
\end{lemma}

Let
\begin{equation}\label{eq:attachment}
	P^+=P\sqcup\int E
\end{equation}
be the disjoint union of $P$ and the element poset $\int E$ of the presheaf $E$.
 
The order on $P^+$ is defined as follows.

\begin{enumerate}
	\item $x,y\in P$ and $x\le y$ in $P$;
	\item $x,y\in \int E$ and $x\le y$ in $\int E$;
	\item  $x\in P$, $y\in E_F$, and $x\in \pi_F(y)$.
\end{enumerate}

\begin{lemma}\label{is-order}
	$(P^+;\le)$ is a poset. 
\end{lemma}
\begin{proof}
	The reflexivity and antisymmetry are obvious. Let $x\le y \le z$ in $P^+$. Suppose $x\le y$ in $P$ and $y\in \pi_F(z)$. Since $\pi_F(z)$ is a lower set, $x\in \pi_F(z)$. Thus $x\le z$ in $P^+$. Suppose $x\in \pi_F(y)$ and $y\le z$ in $\int E$. By the naturality of $\pi$,
	$$x\in \pi_F(y)=\pi_F(\rho^G_F(z))=\pi_G(z)\cap F\subseteq \pi_G(z).$$
	It follows that $x\le z$ in $P^+$. We have proved all non-trivial cases. Therefore, transitivity also holds.
\end{proof}

\begin{lemma}\label{new-label-poset}
	$P^+$ is a labelled poset over $D$. 
\end{lemma}
\begin{proof}
	Assume $x< y$ in $P^+$. The case that $x,y\in P$ is obvious. 
	
	If $x\in P$ and $y\in E_F$ for some $F\in \cF$, then $x\in \pi_F(y)=:S$. Note that $y\in E_F[S,\mathbf e]$ for some compatible tuple $\mathbf e$. Then 
	\begin{align*}
		& K_y\subseteq \bigcup_{e\in E_F[S,\mathbf e]}K_e=\bigcap_{p\in S}K_p\cap \bigcap_{I<F}K_{e_I}\subseteq K_x.
	\end{align*} 
	It is obvious that $a_y\ge a_x$ since $x\in S$.
	
    If $x\in E_F$ and $y\in E_G$ for some $F\subsetneq G$ in $\cF$, then $\rho^G_F(y)=x$. There is a unique $S\in \cT(G)$ and compatible tuple $\mathbf e$ such that $y\in E_G[S,\mathbf e]$. By the construction, we have $x=e_F$. So
    \begin{align*}
    	& K_y\subseteq \bigcup_{e\in E_G[S,\mathbf e]}K_e=\bigcap_{p\in S}K_p\cap \bigcap_{I<G}K_{e_I}\subseteq K_{e_F}=K_x.
    \end{align*}
    It is clear that $a_y\ge a_{e_F}=a_x$. The proof is complete.
\end{proof}

We next investigate minimal upper bounds in $P^+$.

\begin{lemma}\label{property M}
	The following assertions hold.
	\begin{enumerate}
		\item We have $\Min P^+=\Min P$. In addition, for each $x\in P^+$, there is some $p\in \Min P$ with $p\le x$; 
		\item For $p,q\in P$, let $F$ be the mub-closed set generated by $\{p,q\}$. Then
		\begin{equation}\label{mub1}
			\mub_{P^+}\{p,q\}=\mub_P\{p,q\}\cup \{x\in E_F: p,q\in \pi_F(x),\pi_F(x)\cap \mub_P\{p,q\}=\varnothing\}.
		\end{equation}
		For each $x\in P^+$ above $p,q$, there is some $r\in \mub_{P^+}\{p,q\}$ below $x$.
		\item For $p\in P$ and $x\in E_F$ with $F\in \cF$, let $H:=\Cl(\{p\}\cup F)$, i.e., the mub-closed set generated by $\{p\}\cup F$. Then 
		\begin{equation}\label{P&E_F}
			\mub_{P^+}\{p,x\}=\{y\in E_H: p\in \pi_H(y), \rho^H_F(y)=x\}.
		\end{equation}
		Each upper bound $z\in P^+$ of $p$ and $x$ is above some elements of $\mub_{P^+}\{p,x\}$.
		\item  For $x,y\in \int E$, we have 
		\begin{equation}\label{all-in-intE}
			\mub_{P^+}\{x,y\}=\mub_{\int E}\{x,y\}.
		\end{equation}
		Each upper bound of $x,y$ in $P^+$ is above some elements in $\mub_{P^+}\{x,y\}$. 
	\end{enumerate}
	As a consequence, $P^+$ has property M.
\end{lemma}
\begin{proof}
(i) Let $x\in P^+$. Suppose $x\in P$. By property M of $P$, there is some $p\in \Min P$ with $p\le x$; Suppose $x\in E_F$ for some $F\in \cF$. By \namedref{Lemma}{property-of-S}, there is some $p\in \Min P$ with $p\in \pi_F(x)$. Hence, $x$ is above some element of $\Min P$. Since no elements in $P$ is above elements in $\int E$, $\Min P\subseteq \Min P^+$. The assertion follows.  

(ii) Suppose $x\in P^+$ and $x\ge p,q$. If $x\in P$, by property M of $P$, $x$ is above some elements in $\mub_P\{p,q\}$. Let $x\in E_G$ for some $G\in \cF$. Suppose $x$ is above some upper bounds of $p,q$ in $P$. It can be seen from the previous discussion that $x$ is above some elements in $\mub_P\{p,q\}$. Suppose not. Since $p,q\in \pi_G(x)\subseteq G$, it follows that $G$ is a mub-closed set containing $p,q$. Hence, $F\subseteq G$. By naturality of $\pi$,
\begin{align*}
	& \{p,q\}\subseteq F\cap \pi_G(x)=\pi_F(\rho^G_F(x)).
\end{align*}  
Since no upper bounds of $p,q$ in $P$ is below $x$, $$\pi_F(\rho^G_F(x))\cap \mub_P\{p,q\}=\varnothing.$$
Note that $\rho^G_F(x)\le x$. Hence any upper bound of $p,q$ is above some elements in the right hand side of  \namedref{Equation}{mub1}.  

It remains to show that if $x\in E_F$, $p,q\in \pi_F(x)$ and $\pi_F(x)\cap \mub_P\{p,q\}=\varnothing$, then $x\in \mub_{P^+}\{p,q\}$. Suppose $p,q\le y\le x$. By $\pi_F(x)\cap \mub_P\{p,q\}=\varnothing$, we get $y\in \int E$. Since $p,q\in \pi_G(y)\subseteq G$, $F\subseteq G$. It follows from $y\le x$ that $G\subseteq F$. Hence $y=x$, and $x$ is minimal.  

(iii) Suppose $z\in P^+$ is an upper bound of $p,x$. Then $z\in E_G$ for some $G\in \cF$, $p\in \pi_G(z)\subseteq G$, $F\subseteq G$, and $\rho^G_F(z)=x$. It follows that $H\subseteq G$. By naturality of $\pi$, 
\begin{align*}
	& p\in H\cap(\pi_G(z))=\pi_H(\rho^G_H(z)).
\end{align*}
In addition, we have 
\begin{align*}
 & \rho^H_F(\rho^G_H(z))=\rho^G_F(z)=x.
\end{align*}
Hence $z$ is above some elements in the right hand side of \namedref{Equation}{P&E_F}.  

Suppose $y\in E_H$, $p\in \pi_H(y)$, and $\rho^H_F(y)=x$. Let $z\in P^+$ with $p,x\le z\le y$. It must be the case that $z\in E_G$ for some $G\in \cF$. Then $p\in \pi_G(z)\subseteq G$, $F\subseteq G$. Hence $H\subseteq G$. But $z\le y$ implies $G\subseteq H$. If follows that $G=H$ and $z=y$. Whence, $y$ is minimal.  

(iv) It is immediate since no elements in $P$ is above elements in $\int E$, and the poset $\int E$ has property M by \namedref{Lemma}{lem:elements}. 

Finally, it follows from \namedref{Lemma}{lem:M} that $P^+$ has property M. 
\end{proof}

We then verify that \(P^+\) still satisfies properties (i)–(iv) of the labelled poset P.

\begin{lemma}\label{special2}
	In $P^+$, the mub-closed set generated by an arbitrary finite subset is finite.
\end{lemma}
\begin{proof}
	Let $T\subseteq P^+$ be finite. Define
	\begin{align*}
		& A:=\bigcup\{H\in \cF: \exists e\in T\setminus P, e\in E_H\}.
	\end{align*}
	Set $F:=\Cl((T\cap P)\cup A)$. Here ``$\Cl$'' is the mub-closure operator taken in $P$. Then $F$ is a finite subset of $P$. We claim that the finite set
	\begin{align*}
		& R_F:=F\cup \bigcup_{G\in \cF,G\subseteq F}E_G
	\end{align*}  
	is a mub-closed set of $P^+$ containing $T$. Then $\Cl_{P^+}(T)$ will be finite.
	
    Suppose $x\le y$ and $y\in R_F$. If $y\in F$, then $x\in F$ since $F$ is lower in $P$. Now assume $y\in E_G$ for some $G\subseteq F$. Then $x\in P$ implies $x\in \pi_G(y)\subseteq G\subseteq F$. Otherwise, $x\in E_H$ for some $H\in \cF$ implies $H\subseteq G\subseteq F$. Hence $x\in R_F$. We have showed that $R_F$ is a lower set. 
	
    Finally, \namedref{Lemma}{property M} implies that $R_F$ contains all minimal elements, and \namedref{Equation}{mub1}, \namedref{Equation}{P&E_F} and \namedref{Equation}{all-in-intE} imply that $R_F$ is closed under the mub-operator.
	\end{proof}
	
	\begin{lemma}\label{mub-equal}
		For $p,q\in T_0$, we have
		\begin{equation}\label{T0: eq}
			\mub_{P^+}\{p,q\}=\mub_P\{p,q\}.
		\end{equation}
	\end{lemma}
	\begin{proof}
	  Let $F$ be the mub-closed set generated by $p,q$. Suppose there is $x\in E_F$, $p,q\in \pi_F(x)$ and $\pi_F(x)\cap \mub_P\{p,q\}=\varnothing$. There is some $d\in D$ with $\pi_F(x)=S_F(d)$. Hence, we obtain 
	  $$d\in K_p\cap K_q=\bigcup_{c\in \mub_P\{p,q\}}K_c.$$
	  Take a $c\in \mub_P\{p,q\}$ with $d\in K_c$. Then $c\in \mub_P\{p,q\}\cap \pi_F(x)$, a contradiction. This proves \namedref{Equation}{T0: eq} by \namedref{Equation}{mub1}.
 	\end{proof}

\begin{lemma}\label{lem:trace-cover}
	For every $F\in\cF$ and $S\in\cT(F)$,
	\begin{equation}\label{eq:trace-cover}
		\bigcup_{\substack{e\in E_F\\\pi_F(e)=S}}K_e
		=\bigcap_{p\in S}K_p.
	\end{equation}
\end{lemma}
\begin{proof}
	It is shown in \namedref{Lemma}{new-label-poset} that $P^+$ is a labelled poset over $D$. For each $e\in E_F$ with $\pi_F(e)=S$ and each $p\in S$, we have $p\le e$ in $P^+$. Hence $K_e\subseteq K_p$. It shows the ``$\subseteq$'' -part. 
	
	To show the reverse, let $x\in \bigcap_{p\in S}K_p$. Since $x\in \bigcap_{p\in S\cap F_0}K_p$, there is some $e_{F_0}\in E_{F_0}$ with 
	\begin{gather}
		x\in K_{e_{F_0}}\ \ \text{and}\ \  \pi_{F_0}(e_{F_0})=S\cap F_0. 
	\end{gather}
	Assume $I\in (\cF_1\setminus \cF_0)\cap F^{<}$. Then $x\in (\bigcap_{p\in S\cap I}K_p)\cap K_{e_{F_0}}$. There is some $e_I\in E_I[S\cap I,e_{F_0}]$ such that $x\in K_{e_I}$. Note that $\pi_I(e_I)=S\cap I$. Repeat this process, we obtain a compatible tuple $\mathbf e=(e_I)_{I<F}$ such that $x\in K_{e_I}$ and $\pi_I(e_I)=S\cap I$ for all $I<F$. Hence there is some $e_F\in E_F[S,\mathbf e]$ with $x\in K_{e_F}$. Since $\pi_F(e_F)=S$, it proves the ``$\supseteq$''-part.  
\end{proof}

\begin{lemma}\label{special4}
	For all $p,q\in P$, it holds that 
	\begin{equation}\label{special1eq}
		K_p\cap K_q=\bigcup_{c\in \mub_{P^+}\{p,q\}}K_c.
	\end{equation}
\end{lemma}
\begin{proof}
	For each $c\in \mub_{P^+}\{p,q\}$, since $p,q\le c$ we have $K_c\subseteq K_p,K_q$ by \namedref{Lemma}{new-label-poset}. It proves the ``$\supseteq$''-part. 
	
	Let $F$ be the mub-closed subset of $P$ generated by $p,q$. For $x\in K_p\cap K_q$, put $S=S_{F}(x)$. If $S\cap\mub_P\{p,q\}\ne\varnothing$, take $c\in S\cap\mub_P\{p,q\}$. Then $c\in \mub_{P^+}\{p,q\}$ and $x\in K_c$. Otherwise \namedref{Lemma}{lem:trace-cover} supplies
	$e\in E_{F}$ with $\pi_F(e)=S$ and $x\in K_e$. Note that $e\in \mub_{P^+}\{p,q\}$ by \namedref{Equation}{mub1}. This proves the claim.
\end{proof}

We conclude that $P^+$ be a labelled poset over $D$ satisfying the following conditions:
\begin{enumerate}
	\item $P^+$ has property M;
	\item $\Cl(F)$ is finite for each finite subset $F\subseteq P^+$;
	\item $\bigcup_{r\in \Min P^+}K_r=D$;
	\item $P\subseteq P^+$ is a distinguished lower subset whose pairs
	already satisfy
	\begin{equation}\label{eq:settled}
		K_p\cap K_q=\bigcup_{c\in\mub_{P^+}\{p,q\}}K_c
		\qquad(p,q\in P).
	\end{equation}
\end{enumerate}

Now we formulate the initial labelled poset $P_0$.  Set
\begin{equation}\label{initial}
	L_0=\{(r,u,v): r\in \Min D, u,v\in D, r\le u\ll v\}.
\end{equation}
Put 
\begin{equation}\label{P0}
	P_0:=\Min D\sqcup L_0.
\end{equation}
 Define an order $\le$ on $P_0$ by 
 \begin{equation}
 	\le=\{(x,x): x\in P_0\}\cup\{(r,(r,u,v)): r\in \Min D\}.   
 \end{equation}

Notice that incomparable pairs in $P_0$ has no upper bounds. The following assertion is easy to verify.

 \begin{lemma}\label{initial:M}
 	$P_0$ is a poset with property M.
 \end{lemma}

Assign
\begin{equation}
	a_r=r,\qquad K_r={\up r},\qquad a_{(r,u,v)}=u,\qquad K_{(r,u,v)}={\up v}.
\end{equation}
Then $P_0$ becomes a labelled poset over $D$. 

\begin{lemma}\label{generate:fintie}
The set $\Cl(F)$ is finite for each finite subset $F\subseteq P_0$.
\end{lemma}
\begin{proof}
	It is clear that $F\cup \Min D$ is a downward closed finite set containing $F$ and $\Min D$. Since incomparable pairs in $P_0$ has no upper bounds, $F\cup \Min D$ is mub-closed. Hence, $\Cl(F)=F\cup \Min D$ is finite.  
\end{proof}

It is obvious that  $\bigcup_{r\in \Min P_0}K_r=D$. If we set $T_0=\varnothing$, then all of the conditions (i)--(iv) of the initial labelled poset hold. 

	
	

Now we can show the main result of this section.

\begin{theorem}\label{major1}
Every Lawson-compact domain $D$ is a Scott-continuous retract of some Lawson-compact algebraic domain. In detail, there is a poset $P$ with property M such that each principal ideal of $P$ is finite, and $D$ is a Scott-continuous retract of $\Idl(P)$. 
\end{theorem}
\begin{proof}
	Consider the labelled poset $P_0$ and the set $P_{-1}=T_0=\varnothing$. As constructed above, extend $P_0$ to a labelled poset $P_1:=(P_0)^+$ such that every pair in $P_0$ satisfies \namedref{Equation}{eq:settled}. Again, extend $P_1$ to a labelled poset $P_2:=(P_1)^+$ such that every pair in $P_1$ satisfies \namedref{Equation}{eq:settled}. Repeat this process, we obtain a chain
	\begin{align*}
		P_0\subseteq P_1\subseteq P_2\subseteq\ldots
	\end{align*}
	of labelled posets. Put 
	$$P=\bigcup_{n\in \mathbb{N}_0}P_n.$$
	The order on $P$ is defined by 
	\begin{align*}
		& x\le y\Leftrightarrow \exists n\in \mathbb{N}_0, x,y\in P_n, x\le y\text{\ in\ }P_n.
	\end{align*}
	For each $x\in P$, there is some $n\in \mathbb{N}_0$ with $x\in P_n$. We put $(a_x,K_x)$ be the label of $x$ in $P_n$. Then the order and the labels are well-defined, and $P$ becomes a labelled poset over $D$. 
	
	For each $n\in \mathbb{N}_0$, $P_n$ has property M and is an upper set generated by $\Min P_n=\Min P_0=\Min D$. So $\Min P=\Min D$ is a finite set, and each element $x\in P$ is above some elements in $\Min P$. For 
	$x,y\in P$, there is some $n\in \mathbb{N}_0$ with $x,y\in P_n$. By \namedref{Lemma}{mub-equal}, it holds that
	\begin{align*}
		\mub_{P_{n+1}}\{x,y\}=\mub_{P_{n+2}}\{x,y\}=\mub_{P_{n+3}}\{x,y\}=\ldots
	\end{align*}
	Assume $z$ is an upper bound of $x,y$ in $P$. Then $z\in P_s$ for some $s\in \mathbb{N}_0$ with $s\ge n+1$. So $z$ is above some elements in $\mub_{P_{n+1}}\{x,y\}$. Assume $w\in \mub_{P_{n+1}}\{x,y\}$, $p\in P$ and $x,y\le p\le w$. Since no elements in $P_{l+1}\setminus P_l$ is below elements in $P_l$, it must be the case that $p\in P_{n+1}$. It forces $p=w$, hence $w\in \mub_P\{x,y\}$. It follows from \namedref{Lemma}{lem:M} that $P$ has property M.
	
	For (ii) of \namedref{Proposition}{prop:reduction}, it is clear that
	$$\bigcup_{r\in \Min P}K_r=\bigcup_{r\in \Min D}{\up r}=D.$$
	
	For any $x,y\in P$, take a sufficiently large $n\in \mathbb{N}_0$ with $x,y\in P_n$. By the construction, it holds that 
		\begin{align*}
	&	K_x\cap K_y=\bigcup_{z\in\mub_{P_{n+1}}\{x,y\}}K_z=\bigcup_{z\in\mub_{P}\{x,y\}}K_z.
	\end{align*}
	This proves (iii) of \namedref{Proposition}{prop:reduction}. 
	
	Suppose $y\ll x$ in $D$. By the interpolation property, choose $u,v\in D$ with $y\ll u\ll v\ll x$. Take $r\in \Min D$ with $r\le y$. Then $(r,u,v)\in P_0\subseteq P$, and 
	\begin{align*}
		& x\in {\Up v}\subseteq \intrs K_{(r,u,v)}=U_{(r,u,v)}, \qquad y\le a_{(r,u,v)}=u. 
	\end{align*} 
	This proves (iv) of \namedref{Proposition}{prop:reduction}. Thus, all of the assumptions (i)--(iv) of \namedref{Proposition}{prop:reduction} are satisfied. 
	
	It remains to show that for each $x\in P$, the principal ideal ${\downarrow_P x}$ is finite. Take $n\in \mathbb{N}_0$ with $x\in P_n$. Since ${\downarrow_{P_n}x}\subseteq\Cl_{P_n}(x)$, the set ${\downarrow_{P_n}x}$ is finite. Note that ${\downarrow_P x}={\downarrow_{P_n}x}$. As a consequence, $P$ is a poset with property M such that every principal ideal is finite, and $D$ is a Scott-continuous retract of the Lawson-compact algebraic domain $\Idl(P)$. 
\end{proof}

\subsection{Applications of the retraction theorem}

This subsection is devoted to applications of the foregoing retraction theorem. 

In \cite{LuanLi2025}, it is proved that the function space from a Lawson-compact algebraic domain to a bifinite domain is a bifinite domain. This result can be transferred to their retracts.

\begin{corollary}
	The function space from a Lawson-compact domain to an RB-domain is an RB-domain.
\end{corollary}
\begin{proof}
Let $X$ be a Lawson-compact domain and $L$ be an RB-domain. By \namedref{Theorem}{major1}, take a Lawson-compact algebraic domain $X'$ such that $X$ is a Scott-continuous retract of $X'$. In addition, take a bifinite domain $L'$ such that $L$ is a Scott-continuous retract of $L'$. Then $[X'\to L']$ is a bifinite domain. Since $[X\to L]$ is a Scott-continuous retract of $[X'\to L']$, it is an RB-domain.
\end{proof}

This result is parallel to the result in \cite{XiXuLawson2016}, which states that the function space from a coherent domain to a pointed RB-domain is a pointed RB-domain. 

We also use the following theorem of Luan and Li~\cite[Theorem~3.5]{LuanLi2025}.

\begin{theorem}\label{thm:LL}
	If $X$ is a bifinite domain and $L$ is a Lawson-compact algebraic domain, then $[X\to L]$ is a Lawson-compact algebraic domain.
\end{theorem}

Taking retracts, we will obtain the following result. 

\begin{theorem}\label{thm:application}
	If $X$ is an RB-domain and $L$ is a Lawson-compact domain, then
	$[X\to L]$ is a Lawson-compact domain.
\end{theorem}
\begin{proof}
	Let $X'$ be a bifinite domain such that $X$ is a Scott-continuous retract of $X'$. By \namedref{Theorem}{major1}, choose a Lawson-compact algebraic domain $L'$ such that $L$ is a Scott-continuous retract of $L'$. It follows from \namedref{Theorem}{thm:LL} that the function space $[X'\to L']$ is a Lawson-compact algebraic domain. Thus, $[X\to L]$ as a Scott-continuous retract of $[X'\to L']$, is a Lawson-compact domain. 
\end{proof}
	
The next result is routine. 

\begin{lemma}\label{lem:top}
	The following assertions hold:
	\begin{enumerate}
		\item Let $A$ be a Lawson-compact algebraic domain. Adding a fresh top element $\top$ to $A$, we obtain $A^\top$. Then $A^\top$ is a Lawson-compact algebraic domain. In addition, $\Komp(A^\top)=\Komp(A)\cup \{\top\}$. 
		\item Let $L$ be a Lawson-compact domain. Adjoin a new bottom $\bot$ to $L$ and obtain $L_\bot$. Then $L_\bot$ is a Lawson-compact domain.
	\end{enumerate}
\end{lemma}

The relationship between the way-below relation on functions and the pointwise way-below relation is studied in \cite{Luan2026}. The next result states that, provided $L$ has a top element, way-below of functions entails pointwise way-below.
	
\begin{lemma}\label{fs:lem:evaluation}
	Let $X,L$ be dcpos and suppose that $L$ has a greatest element $\top$. Suppose $f\ll g$ in $[X\to L]$. Then for each $x\in X$, we have $f(x)\ll g(x)$. 
\end{lemma}
\begin{proof}
	Fix $x\in X$. Assume $D\subseteq L$ is directed and $\sup D\ge g(x)$. For each $d\in D$, define a function $h_d:X\to L$ by
		\[
	h_d(y)=
	\begin{cases}
		d,&y\leq x,\\
		\top,&y\nleq x.
	\end{cases}
	\]
	It is clear that each $h_d$ is Scott-continuous. In addition, $(h_d)_{d\in D}$ is a directed family in $[X\to L]$ whose supremum is denoted by $h$. Then
			\[
	h(y)=
	\begin{cases}
		\sup D,&y\leq x,\\
		\top,&y\nleq x.
	\end{cases}
	\]
	Hence, $\sup_{d\in D}h_d=h\ge g$. There is some $d\in D$ with $f\le h_d$. It follows that $f(x)\le h_d(x)=d$. So $f(x)\ll g(x)$ in $L$. 
\end{proof}

\begin{lemma}\label{fs:lem:finite-lower}
	Let $A$ be an algebraic domain, $\Komp(A)$ be the poset of all compact elements. Suppose for each $k\in \Komp(A)$, the principal ideal ${\downarrow_{\Komp(A)}k}$ is finite. Then for every $k\in\Komp(A)$, we have ${\dn_A k}={\dn_{\Komp(A)}k}$. Furthermore, if $a\ll b$ in $A$, then $a\in \Komp(A)$.  
\end{lemma}
\begin{proof}
	Suppose $k\in \Komp(A)$, $a\in A$ and $a\le k$. Since $A$ is an algebraic domain, there is a directed set $D$ of compact elements of $A$ with $\sup D=a$. Notice that $D\subseteq {\dn_{\Komp(A)}k}$. Hence $D$ is finite and its supremum equals the largest element of $D$. It follows that $a$ belongs to $D$, and is compact. Thus ${\dn_A k}={\dn_{\Komp(A)}k}$. Assume $a\ll b$ in $A$. The algebraicity of $A$ gives a compact $k\in \Komp(A)$ with $a\le k\le b$. So $a$ is compact. 
\end{proof}

The following result generalizes \namedref{Theorem}{thm:application}. 

\begin{theorem}\label{thm:fs-universal}
	For an FS-domain $X$, the following are equivalent:
	\begin{enumerate}
		\item $X$ is an RB-domain;
		\item $[X\to L]$ is a continuous domain for every Lawson-compact
		algebraic domain $L$;
		\item $[X\to L]$ is a Lawson-compact domain for every Lawson-compact
		domain $L$.
	\end{enumerate}
\end{theorem}
\begin{proof}
\namedref{Theorem}{thm:application} gives (i)$\Rightarrow$(iii), while (iii)$\Rightarrow$(ii) is immediate.  

(ii)$\Rightarrow$(i). Choose an approximate identity $(\alpha_i)_{i\in I}$ on $X$ consisting of finitely separated Scott-continuous maps. For each $i$, fix a finite separating set $F_i\subseteq X$. By \namedref{Theorem}{major1}, there is a poset $P$ with property M such that each principal ideal of $P$ is finite, and $X$ is a Scott-continuous retract of $A:=\Idl(P)$. Let $s:X\to A$ be the section and $r:A\to X$ be the retraction, i.e., $s,r$ are Scott-continuous maps with $r\circ s=\id_{X}$. Adding a fresh top element $\top$ to $A$, we obtain a Lawson-compact algebraic domain $A^{\top}$ by \namedref{Lemma}{lem:top}. By assumption, the function space $[X\to A^{\top}]$ is a continuous domain. Let $i:A\to A^\top$ be the inclusion map and $\bar s:=i\circ s$. It is obvious that $\bar s$ is Scott-continuous. 

Let $g\ll\bar s$ in $[X\to A^\top]$. By continuity and interpolation, choose $h\in [X\to A^\top]$ with $g\ll h\ll \bar s$. \namedref{Lemmas}{fs:lem:evaluation} and~\namedref{Lemma}{fs:lem:finite-lower} show that for every $x\in X$, the value $h(x)$ belongs to $\Komp(A)$ and has a finite principal ideal. Since $h$ is Scott-continuous and $\sup_i\alpha_i=\id_X$, the family $(h\alpha_i)_{i\in I}$ is directed with supremum $h$. Choose $i$ such that $g\leq h\alpha_i$. For $x\in X$, select $a\in F_i$ with $\alpha_i(x)\le a \le x$. Then $g(x)\leq h(\alpha_i(x))\leq h(a)$. Hence, we obtain
$$g(X)\subseteq {\dn h(F_i)}=\bigcup_{a\in F_i}{\dn h(a)}.$$ 
Each set on the right is finite, since $h(a)\in \Komp(A)$. Hence $g$ has finite image.

The set $\mathcal G=\{g\in [X\to A^\top]:g\ll\bar s\}$ is directed and has
supremum $\bar s$. Every $g\in\mathcal G$ satisfies $g\leq\bar s$,
so its image lies in $A$. Its corestriction
$g^\circ :X\to A$ is Scott-continuous, because $A$ is a subdcpo of $A^\top$. Then $\{g^\circ: g\in \mathcal{G}\}$ is directed in $[X\to A]$ and $s$ is their supremum. Define $\delta_g=r g^\circ$. These maps have finite image, are Scott-continuous, and satisfy $\delta_g\leq rs=\id_X$. Moreover,
\begin{equation}\label{fs:eq:deflation-identity}
	\sup_{g\in\mathcal G}\delta_g
	=r\circ\Bigl(\sup_{g\in\mathcal G} g^\circ\Bigr)
	=r\circ s=\id_X.
\end{equation}
Thus $(\delta_g)_{g\in\mathcal G}$ is an approximate identity of deflations. It proves that $X$ is an RB-domain.
\end{proof}

In \cite{AbramskyJung1994}, it is asked whether the function space from an FS-domain to a domain is a continuous domain. Shen at el. \cite{ShenHoXiZhao2026} provide an FS-domain $X$ and an algebraic domain $L$ such that $[X\to L]$ is not continuous, thereby answering the above question. They then ask whether the function space from an FS-domain to a Lawson-compact domain is necessarily a domain. We are now in a position to answer this question.
 
\begin{corollary}\label{fs:cor:pointed-counterexample}
	There exists an FS-domain $X$ and a Lawson-compact algebraic domain $L$ such that the function space $[X\to L]$ is not a continuous domain.
\end{corollary}
\begin{proof}
	Let $X$ be the Lawson's planar closed-disk domain. It is shown in \cite{ChenKouLyu2026} that $X$ is an FS-domain which is not an RB-domain. By \namedref{Theorem}{thm:fs-universal}, there is a Lawson-compact algebraic domain $L$ such that $[X\to L]$ is not continuous.
\end{proof}

\quad

\noindent\textbf{Acknowledgements.} This work is supported by the National Natural Science Foundation of China (No. 12231007) and Hunan Provincial Natural Science Foundation of China (Grant No. 2026JJ60113).

\Needspace{8\baselineskip}


\begin{thebibliography}{99}
\addcontentsline{toc}{section}{References}
\small
\setlength{\itemsep}{1.5pt}
\raggedright

\bibitem{Abb26}
M.~Abbadini,
Formal squares over the unit square separate FS-domains from RB-domains,
\emph{arXiv preprint arXiv:2608.08681},
2026.

\bibitem{AbramskyJung1994}
S.~Abramsky and A.~Jung,
Domain theory,
in: S.~Abramsky, D.~M.~Gabbay, and T.~S.~E.~Maibaum (eds.),
\emph{Handbook of Logic in Computer Science}, vol.~3,
Clarendon Press, Oxford, 1994, pp.~1--168.
Updated online version:
\url{https://achimjungbham.github.io/pub/papers/handy1.pdf}.

\bibitem{BanakhBardyla}
T.~Banakh and S.~Bardyla,
\emph{The interplay between weak topologies on topological semilattices},
arXiv:1804.03736v3 (2021).
\url{https://arxiv.org/abs/1804.03736}.

\bibitem{Che26}
Y.~Chen,
Cone domains separate FS-domains from RB-domains,
\emph{arXiv preprint arXiv:2607.02251},
2026.

\bibitem{ChenKouLyu2026}
Y.~Chen, H.~Kou, and Z.~Lyu,
\emph{FS-domains are not always RB-domains},
arXiv:2607.00568 (2026).
\url{https://arxiv.org/abs/2607.00568}.

\bibitem{Erne2009}
M.~Ern\'e,
Infinite distributive laws versus local connectedness and compactness properties,
\emph{Topology Appl.} \textbf{156} (2009), no.~12, 2054--2069.
\doi{10.1016/j.topol.2009.03.029}.

\bibitem{Erne2019}
M.~Ern\'e,
Web spaces and worldwide web spaces: topological aspects of domain theory,
\emph{Log. Methods Comput. Sci.} \textbf{15} (2019), no.~1, article~23, 1--38.
\doi{10.23638/LMCS-15(1:23)2019}.

\bibitem{FPR21}
T.~Fritz, P.~Perrone and S.~Rezagholi,
Probability, valuations, hyperspace: Three monads on {Top} and the support as a morphism,
\emph{Math. Structures Comput. Sci.} \textbf{31} (2021), no.~8, 850--897.
\doi{10.1017/S0960129521000414}.

\bibitem{CLD}
G.~Gierz, K.~H.~Hofmann, K.~Keimel, J.~D.~Lawson, M.~Mislove, and D.~S.~Scott,
\emph{Continuous Lattices and Domains},
Encyclopedia of Mathematics and its Applications, vol.~93,
Cambridge University Press, Cambridge, 2003.

\bibitem{Gou13}
J.~Goubault-Larrecq,
\emph{Non-{H}ausdorff topology and domain theory},
New Mathematical Monographs, vol.~22,
Cambridge University Press, Cambridge, 2013.

\bibitem{GL2022}
J.~Goubault-Larrecq,
\emph{Topological semilattices with small semilattices},
Non-Hausdorff Topology and Domain Theory, electronic supplement, 2022.
\url{https://projects.lsv.ens-paris-saclay.fr/topology/?page_id=4939}
(accessed September~10, 2026).

\bibitem{JiaLiLuan25}
X.~Jia, Q.~Li, and W.~Luan,
Separating continuous domains from algebraic domains,
\emph{Bull. London Math. Soc.},
(2025),
\doi{10.1112/blms.70049}.

\bibitem{Jung1988}
A.~Jung,
\emph{Cartesian Closed Categories of Domains},
doctoral dissertation, Technische Hochschule Darmstadt, 1988;
published as CWI Tract~66, Centrum voor Wiskunde en Informatica,
Amsterdam, 1989.
\url{https://achimjungbham.github.io/pub/papers/diss.pdf}.
Numbered references in the text refer to the dissertation.

\bibitem{KL09}
K.~Keimel and J.~D.~Lawson,
D-completions and the {$d$}-topology,
\emph{Ann. Pure Appl. Logic} \textbf{159} (2009), 
no.~3, 292--306.
\doi{10.1016/j.apal.2008.06.019}.

\bibitem{KouLuo2003}
H.~Kou and M.-K.~Luo,
RW-spaces and compactness of function spaces for L-domains,
\emph{Topology Appl.} \textbf{129} (2003), no.~3, 211--220.
\doi{10.1016/S0166-8641(02)00169-4}.

\bibitem{Law73}
J.~D.~Lawson,
Intrinsic topologies in topological lattices and semilattices,
\emph{Pacific J. Math.} \textbf{44}(2) (1973), 593--602.
\doi{10.2140/pjm.1973.44.593}.


\bibitem{Lawson1974}
J.~D.~Lawson,
Joint continuity in semitopological semigroups,
\emph{Illinois J. Math.} \textbf{18} (1974), no.~2, 275--285.

\bibitem{LawsonXi2015}
J.~D.~Lawson and X.~Xi,
Continuity of function spaces from pointwise directed families of characteristic functions,
\emph{Order} \textbf{32} (2015), no.~3, 409--418.
\doi{10.1007/s11083-014-9341-9}.

\bibitem{LawsonXu2003}
J.~D.~Lawson and L.~Xu,
Maximal classes of topological spaces and domains determined by function spaces,
\emph{Appl. Categ. Structures} \textbf{11} (2003), 391--402.
\doi{10.1023/A:1024493004302}.

\bibitem{LawsonXu2003'}
J.~D.~Lawson and L.~Xu,
When does the class $[\mathcal{A}\to \mathcal{B}]$ consist of continuous domains?,
\emph{Topology Appl.} \textbf{130} (2003), 91--97.
\doi{10.1016/S0166-8641(02)00213-4}.

\bibitem{LiangKeimel1999}
J.~Liang and K.~Keimel,
Compact continuous L-domains,
\emph{Comput. Math. Appl.} \textbf{38} (1999), 81--89.

\bibitem{LiuLiang1996}
Y.~Liu and J.~Liang,
Solutions to two problems of J.~D.~Lawson and M.~Mislove,
\emph{Topology Appl.} \textbf{69} (1996), 153--164.
\doi{10.1016/0166-8641(95)00083-6}.

\bibitem{Luan2026}
W.~Luan,
$\ll$-separating domains, strong-compact spaces and FS-domains,
\emph{Theoret. Comput. Sci.} \textbf{1075} (2026), article~115935.
\doi{10.1016/j.tcs.2026.115935}.

\bibitem{LuanLi2025}
W.~Luan and Q.~Li,
Function spaces of Lawson-compact algebraic domains,
\emph{Topology Appl.} \textbf{366} (2025), article~109295.
\doi{10.1016/j.topol.2025.109295}.

\bibitem{LuanLiSG}
W.~Luan and Q.~Li,
\emph{Supercompactly Generated Lattices and Bifinite Function Spaces},
unpublished manuscript, undated.

\bibitem{MacMoe92}
S.~Mac Lane and I.~Moerdijk,
\emph{Sheaves in Geometry and Logic: A First Introduction to Topos Theory},
Universitext, Springer-Verlag, New York, 1992.

\bibitem{OpenProbTop90}
J.~van Mill and G.~M.~Reed (Eds.),
\emph{Open Problems in Topology},
North-Holland, Amsterdam, 1990.

\bibitem{ShenHoXiZhao2026}
C.~Shen, W.~K.~Ho, X.~Xi, and D.~Zhao,
\emph{Scott Function Spaces under One-Sided FS Assumptions:
Counterexamples, Positive Results, and New Directions},
arXiv:2607.25832 (2026).
\url{https://arxiv.org/abs/2607.25832}.

\bibitem{Smyth1983}
M.~B.~Smyth,
The largest cartesian closed category of domains,
\emph{Theoret. Comput. Sci.} \textbf{27} (1983), 109--119.

\bibitem{XiXuLawson2016}
X.~Xi, L.~Xu, and J.~D.~Lawson,
Function spaces from coherent continuous domains to RB-domains,
\emph{Topology Appl.} \textbf{206} (2016), 148--157.
\doi{10.1016/j.topol.2016.04.001}.

\bibitem{Xu2003}
L.~Xu,
External characterizations of continuous $sL$-domains,
in: G.~Q.~Zhang, J.~Lawson, Y.-M.~Liu, and M.-K.~Luo (eds.),
\emph{Domain Theory, Logic and Computation},
Semantic Structures in Computation, vol.~3,
Kluwer Academic Publishers, Dordrecht, 2003, pp.~137--149.
\doi{10.1007/978-94-017-1291-0\_6}.


\bibitem{ZouLiGuo2018}
Z.~Zou, Q.~Li, and L.~Guo,
A note on the problem when FS-domains coincide with RB-domains,
\emph{Categ. Gen. Algebr. Struct. Appl.} \textbf{8}(1) (2018), 51--60.
\end{thebibliography}
\end{document}